\documentclass{article}

\usepackage[left=2.5cm,centering]{geometry}
\providecommand{\keywords}[1]
{
  \small	
  \textbf{\textit{Keywords -- }} #1
}

\usepackage{hyperref}
\usepackage{amsfonts,amssymb,amsmath,graphicx,authblk,rotating}
\usepackage[capitalize]{cleveref}
\Crefname{hp}{Assumption}{Assumptions}

\newtheorem{theorem}{Theorem}
\newtheorem{lemma}{Lemma}
\newtheorem{remark}{Remark}
\newtheorem{hp}{Assumption}
\newenvironment{proof}{%
\noindent{\bf Proof.\hskip.5em}\ignorespaces}{%
	}

\usepackage{subcaption}
\usepackage{bm}
\usepackage{mathrsfs}
\usepackage{xfrac}    
\usepackage{tikz}
\usepackage{pgfplots}
\pgfplotsset{compat=1.18}
\usepgfplotslibrary{fillbetween}
\usepackage{lscape}
\usepackage{cases}
\usepackage{mathtools}
\usepackage{cancel}
\usepackage[inter-unit-product=\cdot,
            separate-uncertainty=true,
            print-unity-mantissa=false,
            table-alignment-mode=none,
            table-alignment=right,
            exponent-product=\cdot]{siunitx}
\DeclareSIUnit\mmhg{mmHg}

\newcommand{\bdout}{{\rm w}}

\newcommand{\PP}{{\rm el}}
\newcommand{\FF}{{\rm f}}
\renewcommand{\vec}[1]{{\bm{#1}}}

\newcommand{\jj}{{\rm j}}
\newcommand{\kk}{{\rm k}}
\newcommand{\EE}{{\rm E}}
\newcommand{\II}{{\rm I}}
\newcommand{\DD}{{\rm D}}
\newcommand{\NN}{{\rm N}}

\newcommand{\normcoerc}[1]{%
  |\mkern-1.5mu|
   #1
  |\mkern-1.5mu|
}
\newcommand{\seminormcoerc}[1]{%
  |
   #1
  |
}
\newcommand{\normDGd}[1]{\normcoerc{#1}_{{\rm DG},{\rm D}}}
\newcommand{\normDGj}[1]{\normcoerc{#1}_{{\rm DG},{\rm P}_\jj}}

\newcommand{\normDGu}[1]{\normcoerc{#1}_{{\rm DG},{\rm U}}}
\newcommand{\normDGp}[1]{\normcoerc{#1}_{{\rm DG},{\rm P}_\FF}}
\newcommand{\normENporoel}[1]{\normcoerc{#1}_{\PP,t}}
\newcommand{\normENfluid}[1]{\normcoerc{#1}_{\FF,t}}
\newcommand{\normENtau}[1]{\seminormcoerc{#1}_{\vec{\tau},t}}
\newcommand{\normEN}[1]{\normcoerc{#1}_{{\rm EN},t}}
\newcommand{\normENzero}[1]{\normcoerc{#1}_{{\rm EN},0}}

\newcommand{\normcont}[1]{%
  |\mkern-1.5mu|\mkern-1.5mu|
   #1
  |\mkern-1.5mu|\mkern-1.5mu|
}
\newcommand{\normcontD}[1]{\normcont{#1}_{{\rm D}}}
\newcommand{\normcontJ}[1]{\normcont{#1}_{{\rm P}_\jj}}

\newcommand{\normcontU}[1]{\normcont{#1}_{{\rm U}}}
\newcommand{\normcontP}[1]{\normcont{#1}_{{\rm P}_\FF}}

\newcommand{\Div}{{\rm div}}

\newcommand{\averagel}{\{\!\!\{}
\newcommand{\averager}{\}\!\!\}}

\newcommand{\jumpl}{[\![}
\newcommand{\jumpr}{]\!]}

\newcommand{\average}[1]{\averagel#1\averager}

\newcommand{\jump}[1]{\jumpl#1\jumpr}

\newcommand{\spaceW}{\vec{W}^{\rm DG}_h}
\newcommand{\spaceV}{\vec{V}^{\rm DG}_h}
\newcommand{\spaceQ}{Q^{\rm DG}_h}
\newcommand{\spaceQj}{Q^{\rm DG}_{J,h}}

\newlength{\mywidth}

\title{Discontinuous Galerkin method for a three-dimensional coupled fluid-poroelastic model with applications to brain fluid mechanics}

\author{Ivan Fumagalli}

\affil{\small MOX, Department of Mathematics, Politecnico di Milano, piazza Leonardo da Vinci 32, Milan, 20133, Italy}
\date{}

\begin{document}

\maketitle

\begin{abstract}
The modeling of the interaction between a poroelastic medium and a fluid in a hollow cavity is crucial for understanding, e.g., the multiphysics flow of blood and Cerebrospinal Fluid (CSF) in the brain, the supply of blood by the coronary arteries in heart perfusion, or the interaction between groundwater and rivers or lakes.
In particular, the cerebral tissue's elasticity and its perfusion by blood and interstitial CSF can be described by Multi-compartment Poroelasticity (MPE), while CSF flow in the brain ventricles can be modeled by the (Navier-)Stokes equations, the overall system resulting in a coupled MPE-(Navier-)Stokes system.
The aim of this paper is three-fold.
First, we aim to extend a recently presented discontinuous Galerkin method on polytopal grids (PolyDG) to incorporate three-dimensional geometries and physiological interface conditions.
Regarding the latter, we consider here the Beavers-Joseph-Saffman (BJS) conditions at the interface:
these conditions are essential to model the friction between the fluid and the porous medium.
Second, we quantitatively analyze the computational efficiency of the proposed method on a domain with small geometrical features, thus demonstrating the advantages of employing polyhedral meshes.
Finally, by a comparative numerical investigation, we assess the fluid-dynamics effects of the BJS conditions and of employing either Stokes or Navier-Stokes equations to model the CSF flow.
The semidiscrete numerical scheme for the coupled problem is proved to be stable and optimally convergent.
Temporal discretization is obtained using Newmark's $\beta$-method for the elastic wave equation and the $\theta$-method for the remaining equations of the model.
The theoretical error estimates are verified by numerical simulations on a test case with a manufactured solution,
and a numerical investigation is carried out on a three-dimensional geometry to assess the effects of interface conditions and fluid inertia on the system.
\end{abstract}



\section{Introduction}\label{sec:intro}

Many mathematical models of interest in the applications entail the coupling between a poroelastic medium and a fluid flowing in a hollow region outside of the matrix pores: it is the case, e.g., of the interaction between groundwater and surface waterbodies
\cite{layton2002coupling,discacciati2002mathematical}, 
between the blood-perfused cardiac tissue and the coronary flow \cite{michler2013computationally,digregorio2021computational}, or between the interstitial Cerebrospinal Fluid (CSF) in the brain tissue and its flow in the hollow cerebral ventricles
\cite{baledent2001cerebrospinal,causemann2022human}.
The latter is particularly relevant in the modeling of the waste clearance function played by CSF in the development of neurodegenerative diseases such as Alzheimer's \cite{glymphatic2,brennan2023role}.
Moreover, the CSF flow is strongly interconnected with the pulsatility of blood in the cerebral vasculature and capillaries
\cite{baledent2001cerebrospinal,bacyinski2017paravascular}, thus requiring for a Multiple-network Poroelasticity (MPE) model to account for the interstitial CSF and the blood flowing at different spatial scales in the porous tissue
\cite{guo2018subject,corti2022numerical}.

From the standpoint of finite element literature, interest has been paid to fluid-dynamics and poromechanics problems with the development and analysis of several methods, especially in the Discontinuous Galerkin (DG) class, including interior-penalty DG \cite{sun2005symmetric,lipnikov2014discontinuous,li2015high,AMVZ22}, staggered DG 
\cite{kim2013staggered,cheung2015staggered,zhao2021staggered},
and hybrid DG \cite{boffi2016nonconforming,anderson2018arbitrary,botti2021hybrid}.
The coupling of the fluid and poroelastic systems yields a complex multi-physics problem, investigated in the numerical literature in the case of the Biot-Stokes model
\cite{badia2009coupling,zunino2018biot,wen2020strongly,mardal2021accurate},
also in the regime of large deformations \cite{dereims20153d}, or for multilayered porous media coupled with Newtonian fluid flows \cite{bociu2021multilayered,bukavc2015multilayered}.
However, almost no numerical analysis of the discretization of fluid-poromechanics equations with multiple porous compartments can be found in the literature, despite their relatively wide use in the applications
\cite{digregorio2021computational,barnafi2022multiscale,chou2016fully}:
the MPE system - without a coupled fluid problem - has been analyzed in \cite{lee2019mixed,eliseussen2023posteriori,corti2022numerical}, but the numerical analysis of the coupled MPE-Stokes problem can be found only in \cite{fumagalli2024polytopal}.
In the case of applications to brain function, the main challenges are the high complexity of the domain (encompassing intricate folds and tortuous channels) and the paramount role of stress and flow exchanges at the interface between the two physical domains.
Therefore, a particularly suitable choice is the use of the Polytopal Discontinuous Galerkin (PolyDG) method, for three reasons:
\begin{itemize}
\item by supporting general mesh element shapes, it exhibits strong geometrical flexibility and it allows local refinement and hanging nodes;
\item high-order polynomials can be naturally employed in the discretization, thus guaranteeing low dispersion and dissipation errors, which is particularly relevant at the physical interface;
\item interface conditions can be naturally incorporated in the formulation, thanks to element-wise integration by parts.
\end{itemize}

{The PolyDG method inherits many characteristics from \lq\lq classical'' DG methods in terms of parallel implementation on high-performance computing architectures.
Yet, the use of a polyhedral mesh makes it highly beneficial in dealing with geometrically detailed domains in a computationally efficient way.
In this regard, cost-effective agglomeration strategies have been developed and employed to generate polyhedral meshes without hindering the abovementioned advantage \cite{antonietti2024agglomeration,feder2024r3mg,lymph}.}

In this framework, the present work has {four} aims:
\begin{itemize}
    \item expanding and verifying the PolyDG method proposed in \cite{fumagalli2024polytopal} in three-dimensional geometries;
    \item {assessing the computational advantage of employing a polytopal mesh instead of a classical hexahedral one, in the case of a 3D domain with small geometrical details;}
    \item extending the numerical method to include the physically-motivated Beavers-Joseph-Saffman (BJS) interface conditions and analyzing the stability and convergence of the resulting scheme. The BJS conditions play an essential role in practical applications, like in brain fluid mechanics or groundwater flows, because they take into account the friction and relative tangential velocity between the porous medium and the fluid in contact with it
    \cite{saffman1971boundary,causemann2022human};
    \item assessing the fluid-dynamics effects of modeling the CSF by either Stokes or Navier-Stokes equations, in the aforementioned multi-physics scenario and in the typical regime of brain waste clearance. Both fluid models have been used in the literature, but a direct quantitative comparison between the two is missing.
\end{itemize}

The paper is organized as follows.
\cref{sec:model} describes the multiphysics mathematical model and its weak formulation, particularly discussing the imposition of interface conditions.
The PolyDG space discretization method is introduced in \cref{sec:polydg} and its stability and convergence analysis is presented in \cref{sec:apriori}.
Time discretization is introduced in \cref{sec:fullydiscrete}, while \cref{sec:conv} reports some verification tests.
{A quantitative comparison between the polytopal-mesh-based discretization and a hexahedral one is reported in \cref{sec:polyvsstd}, in terms of the computational cost of solving a differential problem over a domain with small inclusions.}
Numerical results in {an idealized geometry retaining the same topology of the brain and ventricles system, and set in a physiological physical regime,} are discussed in \cref{sec:stokes}, while \cref{sec:navier} reports an assessment of the impact of the BJS condition and the advection term in Navier-Stokes equations on {the fluid-dynamics of the system}.

\section{The mathematical model}\label{sec:model}

\begin{figure}
    \centering
    \includegraphics[width=0.45\textwidth]{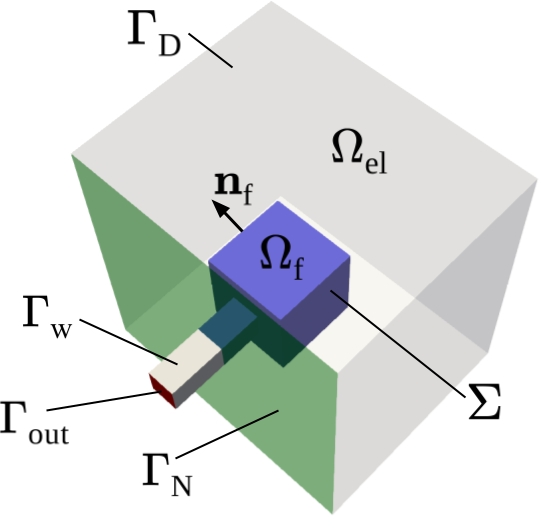}
    \caption{Computational domain: poroelastic region $\Omega_\PP$ and fluid region $\Omega_\FF$, interface $\Sigma$ between them (blue), and external boundaries $\Gamma_\text{out}$ (red), $\Gamma_\text{w}$ (grey), $\Gamma_\text{D}$ (light grey) and $\Gamma_\text{N}$ (green).}
    \label{fig:domain}
\end{figure}

We consider the coupling of a Multiple-Network Poroelasticity system and Navier-Stokes equations in a domain such as the one displayed in \cref{fig:domain}.
The overall domain $\Omega\in\mathbb R^d$ ($d=2,3$) is split into a poroelastic region $\Omega_\PP$ and a fluid region $\Omega_\FF$, separated by the interface $\Sigma=\overline{\Omega}_\PP\cap\overline{\Omega}_\FF$, which is assumed to be a piecewise smooth $(d-1)-$manifold.
The poroelastic region is filled by an elastic solid body and $N_J$ fluid components, indicated by the elements of an index set $J$.
Within this set, let us assume that only one component directly exchanges mass through $\Sigma$: we denote this component by $\EE\in J$.
The displacement of the solid matrix is denoted by $\vec d:\Omega_\PP\to\mathbb R^d$, its corresponding stress tensor is $\sigma_\PP(\vec{d}) = 2\mu_\PP\varepsilon(\vec{d})+\lambda(\nabla\cdot\vec{d})I$, with $\varepsilon(\vec{d})=\frac{1}{2}(\nabla\vec{d}+\nabla\vec{d}^T)$, and each of the fluid components is characterized by a pressure $p_\jj:\Omega_\PP\to\mathbb R, \jj\in J$.
In the fluid domain $\Omega_\FF$, velocity, pressure, and viscous stress are denoted by $\vec{u}$, $p$, and $\tau_\FF(\vec{u})=2\mu_\FF\varepsilon(\vec{u})$, respectively.
The physical parameters in the above definitions and in the following are assumed to be constant.

At the fluid-poroelastic interface $\Sigma$, the following conditions are imposed:
\begin{subnumcases}{\label{eq:interf}}
    \sigma_\PP(\vec{d})\vec{n}_\PP - \sum_{\kk\in J} \alpha_\kk p_\kk\vec{n}_\PP + \tau_\FF(\vec{u})\vec{n}_\FF - p \vec{n}_\FF = \vec{0}, 
    &$ \quad\text{on } \Sigma\times(0,T],$
    \label{eq:BCtotalstress}\\
    p_\EE = p - \tau_\FF(\vec{u})\vec{n}_\FF\cdot\vec{n}_\FF, 
    &$ \quad\text{on } \Sigma\times(0,T],$\label{eq:BCnormalstress}\\
    \frac{1}{\mu_\jj}k_\jj\nabla p_\jj\cdot\vec{n}_\PP = 0,
    &$ \quad\text{on } \Sigma\times(0,T],\quad\forall\jj\in J\setminus\{\EE\},$\label{eq:BCpnonE}\\
    \vec{u}\cdot\vec{n}_\FF + \left(\partial_t\vec{d}-\frac{1}{\mu_\EE}k_\EE\nabla p_\EE\right)\cdot\vec{n}_\PP = 0,
    &$ \quad\text{on } \Sigma\times(0,T],$\label{eq:BCnormalflux}\\
    \left(\tau_\FF(\vec{u})\vec{n}_\FF - p \vec{n}_\FF\right)_{\vec{\tau}} = -\frac{\gamma\mu_\FF}{\sqrt{k_\EE}}(\vec{u}-\partial_t\vec{d})_{\vec{\tau}}
    &$ \quad\text{on } \Sigma\times(0,T],$\label{eq:BCtgstress}
\end{subnumcases}
where $(\vec{\mathtt{v}})_{\vec{\tau}}=\vec{\mathtt{v}}-(\vec{\mathtt{v}}\cdot\vec{n}_\FF)\vec{n}_\FF$ denotes the tangential component of a vector $\vec{\mathtt{v}}\in\mathbb R^d$ along $\Sigma$.
Total stress balance is expressed by condition \eqref{eq:BCtotalstress}, and the normal stress of the fluid at the pores is balanced only by the pressure of compartment $\EE$ (see \eqref{eq:BCnormalstress}).
{
Along the tangential direction, the shear stress is assumed to be proportional to the tangential velocity jump between the fluid and the poroelastic medium (see \eqref{eq:BCtgstress}): this is the Beavers-Joseph-Saffman (BJS) condition, which accounts for the effects of the fluid velocity boundary layer that is created at a fluid-porous medium interface, due to viscous shear stress \cite{beavers1967boundary}.
Notice that, although the condition is written in terms of the total stress in \eqref{eq:BCtgstress}, pressure does not actually play a role in the balance, since $(\vec{n}_\FF)_\vec{\tau}=0$.
The BJS condition has been adopted in the literature to model CSF perfusion interfaces \cite{drosdal2013effect,causemann2022human}, and it can significantly affect the velocity profile near the interface, depending on the value of the slip rate parameter $\gamma$.
To quantitatively assess these effects, in \cref{sec:stokes} we are going to compare the results obtained with the BJS condition against those obtained with the free-slip condition (corresponding to $\gamma=0$) employed, e.g., in \cite{fumagalli2024polytopal}.}

In terms of boundary conditions on $\partial\Omega=(\partial\Omega_\PP\cup\partial\Omega_\FF)\setminus\Sigma$, we consider
a portion $\Gamma_\text{out}\subset(\partial\Omega_\FF\setminus\Sigma)$ of the fluid domain boundary as an outlet and the remaining part $\Gamma_\bdout=\partial\Omega_\FF\setminus(\Sigma\cup\Gamma_\text{out})$ as a solid wall,
and also the poroelastic domain boundary $\partial\Omega_\PP\setminus\Sigma$ is partitioned into a Dirichlet and a Neumann boundary, denoted by $\Gamma_\text{D}$ and $\Gamma_\text{N}$, respectively
(see~\cref{fig:domain}).
Denoting by $T>0$ the final observation time,
the coupled fluid-poroelastic system reads as follows: 
\begin{subnumcases}{\label{eq:NSMPE}}
    \rho_\PP\partial_{tt}^2\vec{d} - \nabla\cdot\sigma_\PP(\vec{d}) + \sum_{\kk\in J}\alpha_\kk\nabla p_\kk = \vec{f}_\PP,
    &$ \qquad\text{in } \Omega_\PP\times(0,T],$ \label{eq:elasticity}\\
    \begin{split}
        c_\jj\partial_t p_\jj+\nabla\cdot\left(\alpha_\jj\partial_t\vec{d}-\frac{1}{\mu_\jj}k_\jj\nabla p_\jj\right) \\
        \qquad+ \sum_{\kk\in J}\beta_{\jj\kk}(p_\jj-p_\kk) + \beta_\jj^\text{e}p_\jj = g_\jj,
    \end{split}
    &$ \qquad\text{in } \Omega_\PP\times(0,T],\quad\forall\jj\in J,$\label{eq:pj}\\
    \rho_\FF\partial_t\vec{u}
    +\rho_\FF(\vec{u}\cdot\nabla)\vec{u}
    - \nabla\cdot\tau_\FF(\vec{u}) + \nabla p = \vec{f}_\FF,
    &$ \qquad\text{in } \Omega_\FF\times(0,T],$ \label{eq:fluidMom}\\
    \nabla\cdot\vec{u} = 0,
    &$ \qquad\text{in } \Omega_\FF\times(0,T],$ \label{eq:fluidCont},\\
    \Bigl(\vec{d}(0), \partial_t\vec{d}(0)\Bigr)=\Bigl(\vec{d}_0,\dot{\vec{d}}_0\Bigr), \quad p_\jj(0)=p_{\jj0},
    &$ \qquad\text{in } \Omega_\PP,\quad\forall \jj\in J,$ \label{eq:mpeinit}\\
    \vec{u}(0) = \vec{u}_0
    &$ \qquad\text{in } \Omega_\FF,$ \label{eq:fluidinit}\\
    \vec{d} = \vec{0},
    \quad
    p_\jj = 0,
    &$ \qquad\text{on } \Gamma_\text{D}\times(0,T],\quad\forall \jj\in J,$\\
    \sigma_\PP(\vec{d})\vec{n} - \sum_{\jj\in J}\alpha_\jj p_\jj\vec{n} = \vec{0},
    \quad
    \frac{1}{\mu_\jj}k_\jj\nabla p_\jj\cdot\vec{n}_\PP = 0,
    &$ \qquad\text{on } \Gamma_\text{N}\times(0,T],\quad\forall \jj\in J,$\\
    \vec{u} = \vec{0},
    &$ \qquad\text{on } \Gamma_\bdout\times(0,T],$\\
    (\tau_\FF(\vec{u}) - p I)\vec{n}_\FF = -\overline{p}^\text{out}\vec{n}_\FF,
    &$ \qquad\text{on } \Gamma_\text{out}\times(0,T],$\\
    \text{and interface conditions \eqref{eq:interf}},
    &$ \qquad\text{on } \Sigma\times(0,T],$
\end{subnumcases}
with suitable definition of the source terms $\vec{f}_\PP:\Omega_\PP\times(0,T]\to\mathbb R^d,g_\jj:\Omega_\PP\times(0,T]\to\mathbb R,\vec{f}_\FF:\Omega_\FF\times(0,T]\to\mathbb R^d$, of the boundary data $\overline{p}^\text{out}:\Gamma_\text{out}\times(0,T]\to\mathbb R$ representing the external normal stress at the outlet, and of the initial conditions $\vec{d}_0:\Omega_\PP\to\mathbb R^d, \dot{\vec{d}}_0:\Omega_\PP\to\mathbb R^d, \vec{u}_0:\Omega_\FF\to\mathbb R^d, p_{\jj 0}:\Omega_\PP\to\mathbb R, \jj\in J$.
Throughout the paper, the data are assumed to be sufficiently regular.

        \begin{remark}[Application to brain fluid-poromechanics]
            The mathematical system \eqref{eq:interf}-\eqref{eq:NSMPE} considered here can be used to model fluid-poromechanics interaction in the brain \cite{corti2022numerical,causemann2022human,fumagalli2024polytopal}, with the fluid domain corresponding to the brain ventricles filled with CSF and the cerebral tissue being the solid matrix perfused by blood and \emph{extracellular} CSF.
            In the fluid compartments $J=\{\text{A},\text{C},\text{V},\EE\}$, $\text{A},\text{C},\text{V}$ can represent the arterial, capillary, and venous blood networks, while the extracellular CSF ($\EE\in J$) is the only compartment exchanging mass with the three-dimensional CSF (cf.~\eqref{eq:BCpnonE}-\eqref{eq:BCnormalflux}), due to the blood-brain barrier.
            Their numerical experiments of that will be presented in \cref{sec:stokes,sec:navier} address this application.
        \end{remark}

The solution variables belong to the following functional spaces:
\[
\mathscr{D}=H^2(0,T; \vec{W}),
\ \ \mathscr{P}=H^1(0,T; [Q_J]^{N_J}),
\ \ \mathscr{V}=H^1(0,T; \vec{V}),\ \ \mathscr{Q}=L^2(0,T;Q),
\]
where the notation $L^2(0,T;H), H^1(0,T; H)$ denotes the time-dependent Bochner spaces associated to a Sobolev space $H$, and
\[\begin{gathered}
\vec{W} = \{\vec{w}\in[H^1(\Omega_\PP)]^d \colon \vec{w}=0 \text{ on }
\Gamma_\text{D}\},\qquad
\vec{V} = \{\vec{v}\in[H^1(\Omega_\FF)]^d\colon \vec{v}=0 \text{ on }\Gamma_\bdout\},\\
Q_J = \{q\in H^1(\Omega_\PP)\colon q=0 \text{ on }\Gamma_{\text{D}}\},\qquad
Q = L^2(\Omega_\FF),
\end{gathered}\]
where $H^1(\Omega)$ denotes the classical Sobolev space of order 1 over $L^2(\Omega)$.

The weak formulation of problem \eqref{eq:NSMPE} reads as follows:\\
Find $(\vec{d},\{p_\jj\}_{\jj\in J},\vec{u},p)\in \mathscr{D}\times\mathscr{P}\times\mathscr{V}\times\mathscr{Q}$ such that, for all $t\in(0,T]$,
\begin{equation}\label{eq:weak}\begin{aligned}
    &(\rho_\PP\partial_{tt}^2\vec{d},\vec{w})_{\Omega_\PP} + a_\PP(\vec{d},\vec{w}) + \sum_{\jj\in J} b_\jj(p_\jj,\vec{w}) - F_\PP(\vec{w}) \\
    &\qquad+ \sum_{\jj\in J}\Bigl[(c_\jj\partial_t p_\jj,q_\jj)_{\Omega_\PP} + a_\jj(p_\jj,q_\jj) + C_j(\{p_\kk\}_{\kk\in J},q_\jj) - b_\jj(q_\jj,\partial_t\vec{d}) - F_\jj(q_\jj) \Bigr]\\
    &\qquad+ (\rho_\FF\partial_t\vec{u},\vec{v})_{\Omega_\FF} + a_\FF(\vec{u},\vec{v})
    + N_\FF(\vec{u},\vec{u},\vec{v})
    + b_\FF(p,\vec{v})+b_\FF(q,\vec{u}) - F_\FF(\vec{v})\\
    &\qquad+ \mathfrak J(p_\EE,\vec{w},\vec{v}) - \mathfrak J(q_\EE,\partial_t\vec{d},\vec{u}) + \mathfrak{G}(\vec{u}-\partial_t\vec{d}, \vec{v}-\vec{w}) = 0\\
\end{aligned}\end{equation}
for all $(\vec{w},\{q_\jj\}_{\jj\in J},\vec{v},q) \in \mathscr{D}\times\mathscr{P}\times\mathscr{V}\times\mathscr{Q}$,
with $\vec{d}(0)=\vec{d}_0, \partial_t\vec{d}(0)=\dot{\vec{d}}_0, \vec{u}(0)=\vec{u}_0, p_\jj(0)=p_{\jj0}$ $\forall \jj\in J$.
In \eqref{eq:weak}, we denoted by $(\cdot,\cdot)_{\Omega}$ the $L^2$-product over $\Omega$
{while the definitions of the bilinear forms are reported in \cref{sec:forms}.}

We point out that, differently from the models studied in \cite{corti2022numerical,fumagalli2024polytopal}, here we consider the following additional terms: the trilinear form $N_\FF$ and the interface form $\mathfrak G$, discussed in the following remarks.

\begin{remark}[Skew-symmetry of the advection form $N_\FF$]\label{rem:adv}
    In the trilinear form 
    {\[
    N_\FF(\vec{u}',\vec{u},\vec{v}) = \left(\rho_\FF(\vec{u}'\cdot\nabla)\vec{u}+\frac{\rho_\FF}{2}(\nabla\cdot\vec{u}')\vec{u},\vec{v}\right)_{\Omega_\FF},
    \]}%
    {we include} the additional term $\left(\frac{\rho_\FF}{2}(\nabla\cdot\vec{u}'),\vec{v}\right)_{\Omega_\FF}$, classically employed for Navier-Stokes problems \cite{temam2001navier}.
    This term vanishes if $\vec{u}'$ is the fluid velocity $\vec{u}$ of \eqref{eq:NSMPE}, but it ensures that that $N_\FF$ is skew-symmetric w.r.t.~exchanging the second and third argument also if $\vec{u}'$ is such that $\nabla\cdot\vec{u}'\neq 0$, as it may occur after the discretization of the equations.
\end{remark}

\begin{remark}[Derivation of the interface forms $\mathfrak J$ and $\mathfrak G$]\label{rem:J}
The interface forms 
{\[
\begin{gathered}
\mathfrak J(p_\EE,\vec{w},\vec{v}) = \int_\Sigma p_\EE\left(\vec{w}\cdot\vec{n}_\PP+\vec{v}\cdot\vec{n}_\FF\right)d\Sigma,\qquad
\mathfrak G(\vec{z}_1, \vec{z}_2) = \int_\Sigma \frac{\gamma\mu_\FF}{\sqrt{k_\EE}}\left(\vec{z}_1\right)_\vec{\tau}\cdot\left(\vec{z}_2\right)_\vec{\tau}d\Sigma.
\end{gathered}\]}%
naturally arise during the derivation of the weak form of problem \eqref{eq:NSMPE}.
We test \eqref{eq:elasticity}-\eqref{eq:pj} against functions $\vec{w}\in\vec{W}$ and
$q_\jj\in Q_J$,
with $\jj\in J$, over $\Omega_\PP$, and \eqref{eq:fluidMom} against $\vec{v}\in \vec{V}$ over $\Omega_\FF$. Then, integrating by parts and summing all the contributions yield the following boundary terms on the interface:
\begin{equation}\label{eq:neumanntermsCONT}
    \int_\Sigma\left[
    (pI-\tau_\FF(\vec{u}))\colon \vec{v}\otimes\vec{n}_\FF
    + \left(\sum_{\kk\in J}\alpha_\kk p_\kk I-\sigma_\PP(\vec{d})\right)\colon \vec{w}\otimes\vec{n}_\PP
    - \sum_{\jj\in J}\frac{1}{\mu_\jj}k_\jj\nabla p_\jj\cdot q_\jj\vec{n}_\PP
    \right]d\Sigma.
\end{equation}
Using the interface conditions \eqref{eq:BCtotalstress},\eqref{eq:BCpnonE} and then \eqref{eq:BCnormalstress},\eqref{eq:BCnormalflux},\eqref{eq:BCtgstress}, we can rewrite \eqref{eq:neumanntermsCONT} as follows:
\begin{equation}
\begin{aligned}
\int_\Sigma&\left[
(pI-\tau_\FF(\vec{u}))\colon (\vec{v}\otimes\vec{n}_\FF + \vec{w}\otimes\vec{n}_\PP) - \frac{1}{\mu_\EE}k_\EE\nabla p_\EE\cdot q_\EE\vec{n}_\PP
\right]d\Sigma \\
&= \int_\Sigma\left[
p_\EE (\vec{v}\cdot\vec{n}_\FF + \vec{w}\cdot\vec{n}_\PP) + \frac{\gamma\mu_\FF}{\sqrt{k_\EE}}(\vec{u}-\partial_t\vec{d})_\vec{\tau} \cdot (\vec{v}-\vec{w})_\vec{\tau} - q_\EE(\vec{u}\cdot\vec{n}_\FF + \partial_t\vec{d}\cdot\vec{n}_\PP)
\right]d\Sigma \\
&=\mathfrak J(p_\EE,\vec{w},\vec{v}) + \mathfrak{G}(\vec{u}-\partial_t\vec{d}, \vec{v}-\vec{w}) - \mathfrak J(q_\EE,\partial_t\vec{d},\vec{u}),
\end{aligned}\end{equation}
where we also used that $\vec{a}\otimes\vec{b}\colon I=\vec{a}\cdot\vec{b}$ for any $\vec{a},\vec{b}\in \mathbb R^d$, and that $\vec{n}_\FF=-\vec{n}_\PP$ on $\Sigma$.
\end{remark}

\section{Polytopal discontinuous Galerkin semi-discrete formulation}\label{sec:polydg}
In this section, we introduce the space discretization of problem \eqref{eq:weak} by a discontinuous Finite Element method on polytopal grids.

Let $\mathscr T_{h,\PP},\mathscr T_{h,\FF}$ be polytopal meshes discretizing the domains $\Omega_\PP,\Omega_\FF$, respectively.
We define the \emph{faces} of an element $K\in \mathscr T_{h,\PP}\cup\mathscr T_{h,\FF}$ as the $(d-1)$-dimensional entities constituting the intersection of $\partial K$ with either the boundary of a neighboring element or the domain boundary $\partial\Omega$.
For $d=2$ all faces are straight line segments, while for $d=3$ they are generic polygons, in principle: we assume that each of these polygons can be further decomposed into triangles, and we define as \emph{face} each of these triangles.
We collect all faces in each physical domain into $\mathscr F_\PP$ and $\mathscr F_\FF$, and we partition these sets into internal faces $\mathscr F_\PP^\II,\mathscr F_\PP^\II$,
Dirichlet/Neumann faces $\mathscr F_\PP^\DD/\mathscr F_\PP^\NN\subset\partial\Omega_\PP\setminus\Sigma,\mathscr F_\FF^\DD/\mathscr F_\FF^\NN\subset\partial\Omega_\FF\setminus\Sigma$ (as portions of the poroelastic and the fluid domain boundaries, respectively),
and interface faces $\mathscr F^\Sigma\subset\Sigma$.
We assume that the polytopal grids $\mathscr T_{h,\PP},\mathscr T_{h,\FF}$ are geometrically conforming with $\Sigma$, but possibly not mesh-conforming.

Over each mesh $\mathscr T_{h,\star}, \star\in\{\PP,\FF\}$, we introduce the broken Sobolev spaces of order $s$, namely
$H^s(\mathscr T_{h,\star}) = \{q\in L^2(\Omega_\star) \colon q|_K\in H^s(K)\quad\forall K\in\mathscr T_{h,\star}\}$.
Moreover, we define the following piecewise polynomial spaces for a given integer $m\geq 1$:
\begin{gather*}
    X_h^\text{DG}(\mathscr T_{h,\star}) =\{\phi\in L^2(\Omega_\star)\colon\phi|_K\in\mathbb P^m(K)\quad\forall K\in\mathscr T_{h,\star}\}, \qquad \star\in\{\PP,\FF\}\\
    \spaceQj = X_h^\text{DG}(\mathscr T_{h,\PP}),
    \quad
    \spaceQ = X_h^\text{DG}(\mathscr T_{h,\FF}),
    \quad
    \spaceW = [X_h^\text{DG}(\mathscr T_{h,\PP})]^d,
    \quad
    \spaceV = [X_h^\text{DG}(\mathscr T_{h,\FF})]^d.
\end{gather*}

To introduce the PolyDG discretization of \eqref{eq:weak}, we define the symmetric outer product $\vec{v}\odot\vec{n} = \frac{1}{2}(\vec{v}\otimes\vec{n}+\vec{n}\otimes\vec{v})$ and, for regular enough scalar-, vector- and tensor-valued functions $q,\vec{v},\tau$,
we define the following average and jump operators.
\begin{itemize}
    \item On each internal face $F\in\mathscr F^\II=\mathscr F_\PP^\II\cup \mathscr F_\FF^\II$ we set:
    \begin{align*}
    \average{q} &= \frac{1}{2}(q^++q^-),
    &\average{\vec{v}} &= \frac{1}{2}(\vec{v}^++\vec{v}^-),
    &\average{\tau} &= \frac{1}{2}(\tau^++\tau^-),
        \\
    \jump{q} &= q^+\vec{n}^++q^-\vec{n}^-,
    &\jump{\vec{v}} &= \vec{v}^+\odot\vec{n}^++\vec{v}^-\odot\vec{n}^-,
    &\jump{\tau} &= \tau^+\vec{n}^++\tau^-\vec{n}^-.
    \end{align*}
    where $\vec{n}^+,\vec{n}^-$ are defined as in \cref{fig:FSigma} - left.
    \item On a Dirichlet face $F\in\mathscr F_\PP^\DD\cup\mathscr F_\FF^\DD$ we set:
    \begin{align*}
    \average{q} &= q,
    &\average{\vec{v}} &= \vec{v},
    &\average{\tau} &= \tau,
        \\
    \jump{q} &= q\vec{n},
    &\jump{\vec{v}} &= \vec{v}\odot\vec{n},
    &\jump{\tau} &= \tau\vec{n},
    \end{align*}
    where $\vec{n}$ is the unit normal vector pointing outward to the element $K$ to which the face $F$ belongs.
    \item On a face $F\in\mathscr F^\Sigma$ shared by two elements $K_\PP\in\mathscr T_{h,\PP}$ and $K_\FF\in\mathscr T_{h,\FF}$ we set:
    \begin{align*}
        \average{q} &= q|_{K_\PP},
        \quad&\average{\tau} &= \tau|_{K_\PP},\\
        \jump{\vec{w},\vec{v}} &= \vec{w}|_{K_\PP}\odot\vec{n}_\PP+\vec{v}|_{K_\FF}\odot\vec{n}_\FF,
        \quad&\jump{\vec{w},\vec{v}}_\vec{\tau} &= (\vec{v}|_{K_\FF})_\vec{\tau} - (\vec{w}|_{K_\PP})_\vec{\tau},
    \end{align*}
    where $\vec{n}_\PP,\vec{n}_\FF$ are defined as in \cref{fig:FSigma} - right.
    Notice that the definitions of the interface jump operators account for the different physics defined on each side of the interface.
\end{itemize}

\begin{figure}
    \centering
    \includegraphics[height=0.27\textwidth]{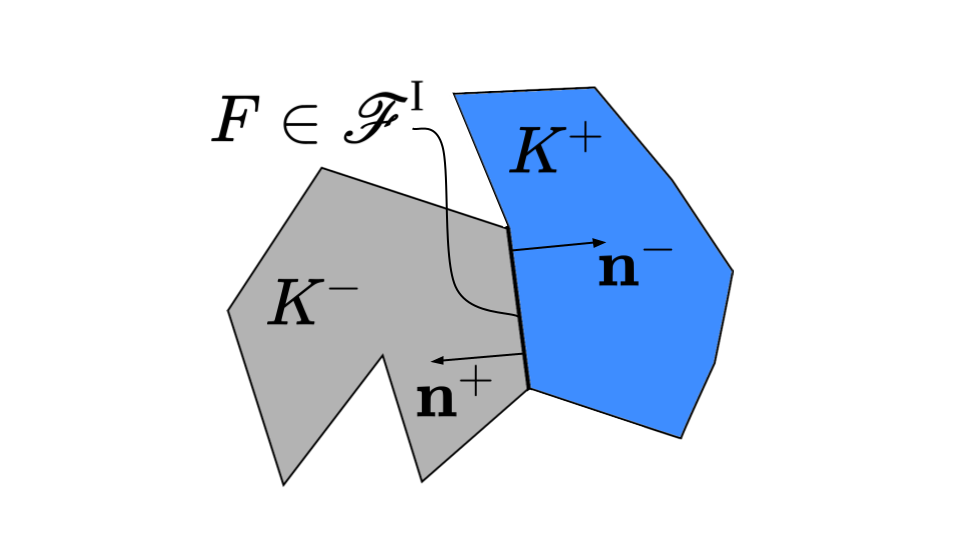}
    \includegraphics[height=0.27\textwidth]{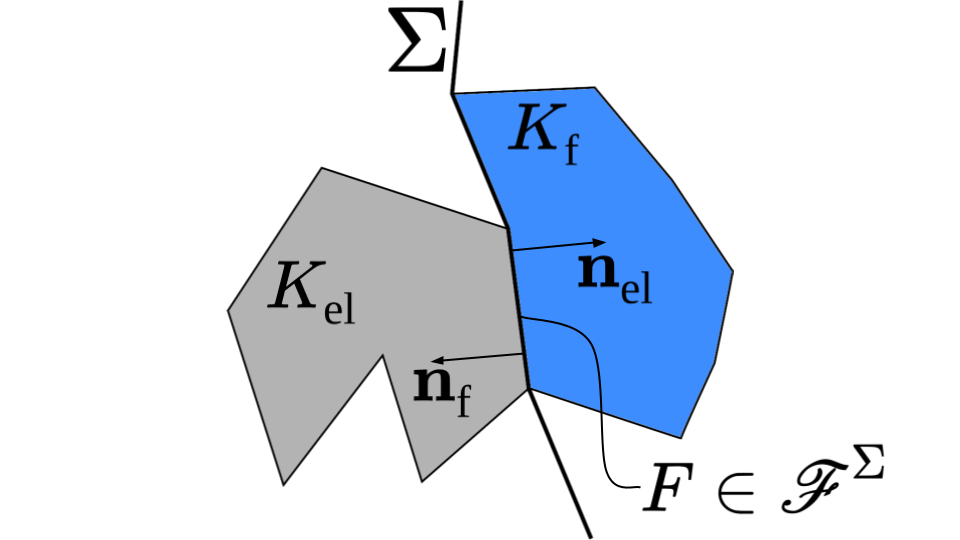}
    \caption{Polygonal elements sharing an internal face (left) or a face on the interface $\Sigma$ (right).}
    \label{fig:FSigma}
\end{figure}


Based on the spaces and the trace operators defined above, the semidiscrete formulation of problem \eqref{eq:weak} reads as follows:
\\
    For any $t\in(0,T]$, find $(\vec{d}_h,\{p_{\jj,h}\}_{\jj\in J},\vec{u}_h,p_h)\in \spaceW\times
    [\spaceQj]^{N_J}
    \times \spaceV \times \spaceQ$ such that
\begin{equation}\label{eq:DG}\begin{aligned}
    (\rho_\PP\partial_{tt}^2\vec{d}_h&,\vec{w}_h)_{\Omega_\PP} + \mathcal L_\PP(\vec{d}_h,\{p_{\kk,h}\}_{\kk\in J}; \vec{w}_h) - \mathcal F_\PP(\vec{w}_h) \\
    & + \sum_{\jj\in J} \left[ (c_\jj\partial_t p_{\jj,h}, q_{\jj,h})_{\Omega_\PP} + \mathcal L_\jj(\{p_{\kk,h}\}_{\kk\in J},\partial_t\vec{d}_h;q_{\jj,h}) - \mathcal F_\jj(q_{\jj,h}) \right] \\
    & + (\rho_\FF\partial_t\vec{u}_h,\vec{v}_h)_{\Omega_\FF}+\mathcal L_\FF(\vec{u}_h,p_h;\vec{v}_h,q_h) - \mathcal F_\FF(\vec{v}_h)
    \\
    & + \mathcal J(p_{\EE,h}, \vec{w}_h,\vec{v}_h) - \mathcal J (q_{\EE,h}, \partial_t\vec{d}_h,\vec{u}_h) + \mathcal G(\vec{u}_h-\partial_t\vec{d}_h, \vec{v}_h-\vec{w}_h)= 0 \\ &\forall\vec{w}_h\in\spaceW,\vec{v}_h\in\spaceV,q_h\in\spaceQ,q_{\jj,h}\in\spaceQj,
\end{aligned}\end{equation}
with initial conditions defined in terms of the projections $\vec{d}_h(0), \dot{\vec{d}}_h(0), \{p_{\jj,h}(0)\}_{\jj\in J}, \vec{u}_h(0)$ of the initial data introduced in \eqref{eq:NSMPE} onto the corresponding DG spaces.
The forms and functionals appearing in \eqref{eq:DG} are the PolyDG version of those defined in \eqref{eq:formsContPb} and their complete definitions can be found in \cref{sec:forms}.
Here we just report the definition of the interface terms, to show the role of the interface jump operators:
\begin{align*}
    \mathcal J(p_\EE, \vec{w},\vec{v}) &=
        \sum_{F\in\mathscr{F}^\Sigma}\int_F\left(
    \average{p_\EE I} \colon \jump{\vec{w},\vec{v}}
    \right)
    ,\\
    \mathcal G(\vec{v}_1-\vec{w}_1,\vec{v}_2-\vec{w}_2) &=
        \sum_{F\in\mathscr{F}^\Sigma}\int_F
    \frac{\gamma\mu_\FF}{\sqrt{k_\EE}}\jump{\vec{w}_1,\vec{v}_1}_\vec{\tau} \cdot \jump{\vec{w}_2,\vec{v}_2}_\vec{\tau}.
\end{align*}
Moreover, we point out that the definitions of $\mathcal L_\PP,\mathcal L_\FF, \mathcal L_\jj,\jj\in J,$ include stabilization terms with parameters that depend on the mesh element size (see \cref{sec:forms}), that will be useful in the theoretical analysis of \cref{sec:apriori}.
However, these stabilization terms are defined only on internal faces and do not contribute to the interface terms.

\section{Stability and error analysis of the semidiscrete Stokes-MPE problem}\label{sec:apriori}

In this section, we analyze the semidiscrete problem corresponding to a choice of the Stokes equations to model the CSF flow in $\Omega_\FF$, namely we neglect the nonlinear advection term in the form $\mathcal L_\FF$ of \eqref{eq:DG} (see form $\mathcal N_\FF$ in \eqref{eq:formsAllTogetherFF}).
The following analysis holds for a generic set $J$ made of $N_J\in\mathbb N_0$ fluid compartments of the poroelastic model.
For the sake of simplicity, we make the following assumptions.
\begin{hp}\label{hp}
\
\begin{itemize}
    \item all the physical parameters of the model (cf.~\cref{tab:modelparams}) are piecewise constant over the aforementioned decomposition;
    \item the polytopal mesh $\mathscr T_h=\mathscr T_{h,\PP}\cup\mathscr T_{h,\FF}$ fulfills the regularity assumptions of the PolyDG framework, namely that $\mathscr T_h$ is $h$-uniformly polytopic-regular, a local bounded variation property holds, and there exists a suitable shape-regular simplicial covering of $\mathscr T_h$ \cite{antonietti2016review, cangiani2014hp}.
\end{itemize}
\end{hp}
In all the inequalities appearing hereafter, the dependency on the model parameters and the finite element degree $m$ will be neglected: by $x\lesssim y$ we will indicate that $\exists C>0:x\leq C y$, with $C$ independent of the space discretization parameters.

Following \cite{fumagalli2024polytopal,corti2022numerical,AMVZ22}, we define the following broken norms:
\begin{subequations}\label{eq:brokennorms}\begin{align}
    \normDGd{\vec{d}}^2 &= \|{\mathbb C_\PP^{1/2}}[\varepsilon_h(\vec{d})]\|_{L^2(\mathscr T_{h,\PP})}^2 + \|\sqrt{\eta}\jump{\vec{d}}\|_{\mathscr F_{\PP,h}^{\II}\cup\mathscr F_{\PP,h}^{\DD}}^2
    &\forall\vec{d}\in \vec{H}^1({\mathscr T_{h,\PP}}),
    \\
    \normDGj{p}^2 &= \|\mu_\jj^{-1/2}k_\jj^{1/2}\nabla_h p\|_{L^2(\mathscr T_{h,\PP})}^2 + \|\sqrt{\zeta_\jj}\jump{p}\|^2_{\mathscr F_{\PP,h}^{\II}\cup\mathscr F_{\PP,h}^{\DD_\jj}}
    &\forall p \in H^1({\mathscr T_{h,\PP}}),
    \\
    \normDGu{\vec{u}}^2 &= \|\sqrt{2\mu}\,\varepsilon_h(\vec{u})\|_{L^2(\mathscr T_{h,\FF})}^2 + \|\sqrt{\gamma_\vec{v}}\jump{\vec{u}}\|^2_{\mathscr F_{\FF,h}^{\II}}
    &\forall\vec{u}\in \vec{H}^1({\mathscr T_{h,\FF}}),
    \\
    \normDGp{q}^2 &= \|q\|_{L^2(\Omega_\FF)}^2 + \|\sqrt{\gamma_p}\jump{q}\|^2_{\mathscr F_{\FF,h}^{\II}\cup\mathscr F_{\FF,h}^{\DD}}
    &\forall q\in H^1({\mathscr T_{h,\FF}}),
\end{align}\end{subequations}
and we introduce the following energy norms at time $t\in(0,T]$, based on those broken norms but also depending on the tangential velocity of the fluid and of the poroelastic structure along the interface $\Sigma$:
\begin{align}\label{eq:normcoerc}
    \normEN{(\vec{d},\{p_\jj\}_{\jj\in J},\vec{u},p)}
    &= \left[ \normENporoel{(\vec{d},\{p_\jj\}_{\jj\in J})}^2+ \normENfluid{(\vec{u},p)}^2 + \normENtau{(\vec{u},\partial_t\vec{d})}^2\right]^{1/2},
\end{align}
where
\begin{align*}
    \normENporoel{(\vec{d},\{p_\jj\}_{\jj\in J})}
    &= \left[\|\sqrt{\rho_\PP}\partial_t\vec{d}(t)\|_{\Omega_\PP}^2 + \normDGd{\vec{d}(t)}^2
    \phantom{\sum_{\jj\in J}\int_0^t}\right.\\&\qquad\left.
    + \sum_{\jj\in J}\left(\|\sqrt{c_\jj}p_\jj(t)\|_{\Omega_\PP}^2 + \int_0^t\left(\normDGj{p_\jj(s)}^2+\|\sqrt{\beta^\text{e}_\jj}p_\jj(s)\|_{\Omega_\PP}^2\right)ds\right)\right]^{1/2},
    \\
    \normENfluid{(\vec{u},p)}
    &= \left[\|\sqrt{\rho_\FF}\vec{u}(t)\|_{\Omega_\FF}^2 + \int_0^t\left(\normDGu{\vec{u}(s)}^2 +\normDGp{p(s)}^2
    \right)ds\right]^{1/2},
    \\
    \normENtau{(\vec{u},\partial_t\vec{d})}
    &= \left[\int_0^t\mathcal G\left(\vec{u}(s)-\partial_t\vec{d}(s),\vec{u}(s)-\partial_t\vec{d}(s)\right)ds\right]^{1/2}.
\end{align*}
{It can be shown that $ \normEN{(\vec{d},\{p_\jj\}_{\jj\in J},\vec{u},p)}$, although being only a semi-norm over $\mathscr{D}\times\mathscr{P}\times\mathscr{V}\times\mathscr{Q}$, is an actual norm over the semi-discrete space $H^2(0,T;\spaceW)\times H^1(0,T;\spaceQj)\times H^1(0,T;\spaceV)\times L^2(0,T;\spaceQ)$ for any value of $\gamma\geq 0$, thanks to the element-wise regularity of polynomials.
Consequently, the following stability result ensures the uniqueness of the semi-discrete solution of problem \eqref{eq:DG}.}
\begin{theorem}[Stability estimate]\label{th:stab}
Under \cref{hp} and assuming that sufficiently large values are chosen for
the penalty constants (cf.~\eqref{eq:penaltyparams}), the solution $(\vec{d}_h,\{p_{\jj,h}\}_{\jj\in J},\vec{u}_h,p_h)$ of the semidiscrete problem \eqref{eq:DG} fulfills the following inequality for each time $t\in(0,T]$:
\begin{equation}\label{eq:stab}\begin{aligned}
    \normEN{(\vec{d}_h,\{p_{\jj,h}\}_{\jj\in J},\vec{u}_h,p_h)} \lesssim
    & \normENzero{(\vec{d}_h,\{p_{\jj,h}\}_{\jj\in J},\vec{u}_h,0)}
    \\& + \int_0^t\left(\frac{1}{\sqrt{\rho_\PP}}\|\vec{f}_\PP\|_{\Omega_\PP}+\sum_{\jj\in J}\frac{1}{\sqrt{c_\jj}}\|g_\jj\|_{\Omega_\PP}+\frac{1}{\sqrt{\rho_\FF}}\|\vec{f}_\FF\|_{\Omega_\FF}\right)ds,
\end{aligned}\end{equation}
where the first term depends on the initial conditions \eqref{eq:mpeinit}-\eqref{eq:fluidinit}:
\[
\normENzero{(\vec{d}_h,\{p_{\jj,h}\}_{\jj\in J},\vec{u}_h,0)} = \left[\|\sqrt{\rho_\PP}\dot{\vec{d}}_h^0\|_{\Omega_\PP}^2 + \normDGd{\vec{d}_h^0}^2 + \sum_{\jj\in J}\|\sqrt{c_\jj}p_{\jj,h}^0\|_{\Omega_\PP}^2 + \|\sqrt{\rho_\FF}\vec{u}_h^0\|_{\Omega_\FF}^2\right]^{1/2}.
\]
\end{theorem}
\begin{proof}
Choosing the test functions $\vec{w}_h=\partial_t\vec{d}_h(t), \vec{v}_h=\vec{u}_h(t), q_h=p_h(t), q_{\jj,h}=p_{\jj,h}(t)\ \forall\jj\in J$ in the semidiscrete problem \eqref{eq:DG} and following the arguments of \cite[Theorem 4.1]{fumagalli2024polytopal}, we obtain the following inequality:
\[
    \normENporoel{(\vec{d},\{p_\jj\}_{\jj\in J})}^2 
    + \normENfluid{(\vec{u},p)}^2 + \int_0^t\mathcal G\left(\vec{u}_h(s)-\partial_t\vec{d}_h(s),\vec{u}_h(s)-\partial_t\vec{d}_h(s)\right)ds
    \leq \text{RHS},
\]
where $\text{RHS}$ denotes the right-hand side of \eqref{eq:stab}.
Observing that the term with the form $\mathcal G$ coincides with the definition of $\normENtau{(\vec{u}_h,\partial_t\vec{d}_h)}^2$ concludes the proof.
\end{proof}

\begin{remark}
We point out that the semi-discrete problem \eqref{eq:DG} is stable in the Navier-Stokes case, too, since the advection form $N_\FF$ is skew-symmetric and thus cancels out in the proof of \cref{th:stab}.
\end{remark}

\bigskip

To derive an a-priori estimate for the space discretization error of the proposed PolyDG method, we rely on the following additional norms
for non-discrete functions:
\[\begin{aligned}&\begin{aligned}
\normcontD{\vec{w}}^2 &= \normDGd{\vec{w}}^2 + \|\eta^{-1/2}\average{\sigma_\PP(\vec{w})}\|_{\mathscr F_{\PP,h}^\text{I}\cup\mathscr F_{\PP,h}^\text{D}}^2
&\qquad\forall \vec{w}\in [H^2(\mathscr T_{h,\PP})]^d,\\
\normcontJ{q_\jj}^2 &= \normDGj{q_\jj}^2 + \|\zeta^{-1/2}\average{\frac{1}{\mu_\jj}k_\jj\nabla_hq_\jj}\|_{\mathscr F_{\PP,h}^\text{I}\cup\mathscr F_{\PP,h}^{\text{D}_\jj}}^2
&\qquad\forall q_\jj\in H^2(\mathscr T_{h,\PP}),
\qquad \forall\jj\in J,\\
\normcontU{\vec{v}}^2 &= \normDGu{\vec{v}}^2 + \|\gamma_{\vec{v}}^{-1/2}\average{\tau_\FF(\vec{v})}\|_{\mathscr F_{\FF,h}^\text{I}\cup\mathscr F_{\FF,h}^\text{D}}^2
&\qquad\forall \vec{v}\in [H^2(\mathscr T_{h,\FF})]^d,\\
\normcontP{q}^2 &= \normDGp{q}^2 + \|{\gamma_p^{1/2}}\average{q}\|_{\mathscr F_{\FF,h}^\text{I}}^2
&\qquad\forall q\in H^1(\mathscr T_{h,\FF}),
\end{aligned}\\
&\normcont{(\vec{w},\{q_\jj\}_{\jj\in J},\vec{v},p)}^2 = \normcontD{\vec{w}}^2+\sum_{\jj\in J}\normcontJ{q_\jj}^2 + \normcontU{\vec{v}}^2+\normcontP{q}^2.
\end{aligned}\]
Moreover, we denote by $\mathscr E_K: H^s(\Omega)\to H^s(\mathbb R^d)$ the Stein extension operator from a
Lipschitz domain $\Omega$ defined in \cite{stein1970singular}, for which optimal interpolation results can be proven w.r.t.~the norms defined above (cf.~\cref{sec:stein}).
Combining the results above, we can prove the following optimal convergence estimate:
\begin{theorem}[A priori error estimate]\label{th:conv}
    Let us assume that \cref{hp} holds and that the penalty parameters
included in the discrete problem's formulation are sufficiently large (see \eqref{eq:penaltyparams}).
If the solution of problem \eqref{eq:weak}
is sufficiently regular,
    the following estimate holds for each $t\in (0,T]$: 
    \begin{equation}\label{eq:error}
    \begin{aligned}
    &
    \normEN{(\vec{e}^\vec{d},\{e^{p_\jj}\}_{\jj\in J},\vec{e}^\vec{u},e^{p})}^2
    \\ &\qquad
    \lesssim \sum_{K\in\mathscr T_{h,\PP}}
    {h_K^{2m}}\left\{
    \|\mathcal E_K\vec{d}(t)\|_{[H^{m+1}(\widehat{K})]^d}^2 + \sum_{\kk\in J}\|\mathcal E_K p_\kk(t)\|_{H^{m+1}(\widehat{K})}^2
    \right.
    \\ &\qquad
    \qquad\qquad\left.
    +\int_0^t\left[
    \|\mathcal E_K\partial_t\vec{d}(s)\|_{[H^{m+1}(\widehat{K})]^d}^2+\|\mathcal E_K\partial_{tt}^2\vec{d}(s)\|_{[H^{m+1}(\widehat{K})]^d}^2
    \right]ds
    \right.
    \\ &\qquad
    \qquad\qquad\left.
    +\int_0^t\left.
    \sum_{\kk\in J}\left(
    \|\mathcal E_Kp_\kk(s)\|_{H^{m+1}(\widehat{K})}^2+\|\mathcal E_K\partial_tp_\kk(s)\|_{H^{m+1}(\widehat{K})}^2
    \right)
    \right.ds
    \right\}
    \\ &\qquad
    \quad+ \sum_{K\in\mathscr T_{h,\FF}}
    {h_K^{2m}}\left.
    \int_0^t\left[
    \|\mathcal E_K\vec{u}(s)\|_{[H^{m+1}(\widehat{K})]^d}^2 + \|\mathcal E_K\partial_t\vec{u}(s)\|_{[H^{m+1}(\widehat{K})]^d}^2
    \right.\right.\\ &\qquad\qquad\qquad\qquad\qquad\left.\left.\phantom{\|^2_{H^{m+1}}}
    +\|\mathcal E_Kp(s)\|_{H^{m+1}(\widehat{K})}^2
    \right]ds,
    \right.
    \end{aligned}
    \end{equation}
    where $\vec{e}^\vec{d}=\vec{d}-\vec{d}_h, e^{p_\jj}=p_\jj-p_{\jj,h}\ \forall \jj\in J, \vec{e}^\vec{u}=\vec{u}-\vec{u}_h, e^{p}=p - p_h$,
and $\widehat{K}\supseteq K$, for each $K\in\mathscr T_h$, are shape-regular simplexes covering $\mathscr T_h$, as in \cref{hp}.
\end{theorem}
\begin{proof}
Considering the continuous displacement $\vec{d}$ and its semi-discrete counterpart $\vec{d}_h$, we introduce the error splitting $\vec{d}-\vec{d}_h=\vec{e}^\vec{d}_I-\vec{e}^\vec{d}_h$, with $\vec{e}^\vec{d}_I=\vec{d}-\vec{d}_I\in[H^{m+1}(\mathscr T_{h,\PP})]^d$ being the Stein interpolation error (cf.~\cref{sec:stein}, \cref{th:interp}) and $\vec{e}^\vec{d}_h=\vec{d}_h-\vec{d}_I\in\spaceW$ being the approximation error.
Analogous definitions are introduced for $e^{p_\jj}\ \forall\jj\in J, \vec{e}^\vec{u}, e^{p}$.
We subtract the continuous problem \eqref{eq:weak} from the semi-discrete problem \eqref{eq:DG}, tested against $(\partial_t\vec{e}^\vec{d}_h,\{e^{p_{\jj}}_h\}_{\jj\in J},\vec{e}^\vec{u}_h,e^{p}_h)$, we apply the coercivity of the bilinear forms $\mathcal{L}_\PP,\mathcal{L}_\FF,\mathcal{L}_\jj,\jj\in J,$ and the continuity of these forms as well as of {$\mathcal J$} (cf.~\cite[Lemmas 2 and 5]{fumagalli2024polytopal}), and then we can follow the same steps of \cite[Theorem 4.2]{fumagalli2024polytopal} to arrive to the following inequality:
\begin{equation}\label{eq:fromOldPaper}
\begin{aligned}
\normENporoel{(\vec{e}^\vec{d}_h,\{e^{p_\jj}_h\}_{\jj\in J})}^2&+ \normENfluid{(\vec{e}^\vec{u}_h,e^p_h)}^2
+\int_0^t\mathcal G(\vec{e}_h^\vec{u}(s)-\partial_t\vec{e}_h^\vec{d}(s), \vec{e}_h^\vec{u}(s)-\partial_t\vec{e}_h^\vec{d}(s))\,ds\\
&\lesssim
\text{RHS} + \int_0^t\mathcal G(\vec{e}_I^\vec{u}(s)-\partial_t\vec{e}_I^\vec{d}(s), \vec{e}_h^\vec{u}(s)-\partial_t\vec{e}_h^\vec{d}(s))\,ds\\
&\lesssim \text{RHS} + \normENtau{(\vec{e}_I^\vec{u},\partial_t\vec{e}_I^\vec{d})}\normENtau{(\vec{e}_h^\vec{u},\partial_t\vec{e}_h^\vec{d})},
\end{aligned}
\end{equation}
where $\text{RHS}$ is the right-hand side of the thesis \eqref{eq:error}. We have used the linearity of $\mathcal G$ and the fact that the squared energy norm defined in this work corresponds to that of \cite{fumagalli2024polytopal} plus the tangential velocity squared seminorm $\normENtau{(\vec{e}^\vec{u}_h,\partial_t\vec{e}^\vec{d}_h)}^2=\int_0^t\mathcal G(\vec{e}_h^\vec{u}(s)-\partial_t\vec{e}_h^\vec{d}(s), \vec{e}_h^\vec{u}(s)-\partial_t\vec{e}_h^\vec{d}(s))\,ds$.
Now, since the seminorm $\normENtau{(\vec{e}_I^\vec{u},\partial_t\vec{e}_I^\vec{d})}$ can be controlled by the corresponding broken $L^2$ norm over the interface faces $F\in\mathscr F^\Sigma$, we can employ Stein interpolation results and obtain the following (see \cref{sec:stein}):
\[\begin{aligned}
&\normENtau{(\vec{e}_I^\vec{u},\partial_t\vec{e}_I^\vec{d})}
\\&\qquad\lesssim
\int_0^t\left(\|\vec{e}_I^\vec{u}(s)\|_{\mathscr F^\Sigma}+\|\partial_t\vec{e}_I^\vec{d}(s)\|_{\mathscr F^\Sigma}\right)\,ds
\\&\qquad\lesssim
\int_0^t\left(\sum_{\mbox{\scriptsize\ensuremath{\begin{gathered}K\in\mathscr T_{h,\FF}\\[-1ex]\partial K\cap\Sigma\neq\emptyset\end{gathered}}}}h_K^{m+1/2}\|\mathcal E\vec{u}(s)\|_{H^{m+1}(\widehat{K})}+\sum_{\mbox{\scriptsize\ensuremath{\begin{gathered}K\in\mathscr T_{h,\FF}\\[-1ex]\partial K\cap\Sigma\neq\emptyset\end{gathered}}}}h_K^{m+1/2}\|\mathcal E\partial_t\vec{d}(s)\|_{H^{m+1} (\widehat{K})}\right)\,ds,
\end{aligned}\]
where $\widehat{K}\supseteq K$, for each $K\in\mathscr T_h$, are shape-regular simplexes covering $\mathscr T_h$, existing thanks to \cref{hp}.
Then,
the application of Cauchy-Schwarz and Young inequalities on \eqref{eq:fromOldPaper} concludes the proof.
\end{proof}

\begin{remark}
It is worth to point out that, in the proof of \cref{th:conv}, the contribution of the interpolation error in the $\normENtau{\cdot}$ seminorm has a convergence order that is 1/2 greater than the other terms of the energy norm.
This means that the treatment of the BJS condition in the proposed numerical method does not affect convergence.
\end{remark}

\section{Fully discrete problem and implementation}\label{sec:fullydiscrete}
Starting from the semi-discrete problem \eqref{eq:DG}, we introduce a timestep $\Delta t$ and a corresponding time discretization over a uniform partition $\{t^n=n\Delta t\}_{n=0}^N$ of the interval $(0,T]$.
We use Newmark's $\beta$-method to discretize the terms tested against $\vec{w}_h$ in \eqref{eq:DG} (corresponding to the elastic momentum equation) and the Crank-Nicolson method for all the other terms.
The nonlinear advection term is linearized by a semi-implicit approach, using a second-order extrapolation of the advecting velocity at time $t^{n+\sfrac{1}{2}}$:
\begin{equation}\label{eq:linearization}
(\vec{u}\cdot\nabla)\vec{u}\ |_{t=t^{n+\sfrac{1}{2}}} \simeq \left[\left(\frac{3}{2}\vec{u}^n - \frac{1}{2}\vec{u}^{n-1}\right)\cdot\nabla\right]\frac{\vec{u}^{n+1}+\vec{u}^n}{2},
\end{equation}
where the superscript $\cdot^n$ denotes the discrete variable approximating $\vec{u}$ at time $t^n$.
A similar notation is used in the following for all the other variables.

In accordance with the above discretization methods, the  algebraic form of the fully discrete problem reads as follows:
\begin{equation}\label{eq:fullydiscr}
A_1(\vec{U}^{n},\vec{U}^{n-1}) \vec{X}^{n+1} = A_2(\vec{U}^{n},\vec{U}^{n-1})\vec{X}^n + \vec{F}^{n+1}, \qquad n=1,\dots,N,
\end{equation}
where
\begin{align}
\vec{X}^n &= \begin{bmatrix}
\vec{D}^n ; \vec{Z}^n ; \vec{A}^n ; \vec{P}_\text{A}^n ; \dots ; \vec{P}_\EE^n ; \vec{U}^n ; \vec{P}^n
\end{bmatrix},
\ 
\vec{F}^{n+1} = \begin{bmatrix}
\vec{F}_\PP^n ;
\vec{0} ;
\vec{0} ;
\vec{F}_\text{A}^\text{CN}; \dots ; \vec{F}_\EE^\text{CN} ; \vec{F}_\FF^\text{CN}; \vec{0}
\end{bmatrix}.
\end{align}
To employ Newmark's scheme, we have introduced two auxiliary vector variables $\vec{Z}^n, \vec{A}^n$ representing the first and second time derivatives of $\vec{D}^n$.
The definition of the matrices $A_1(\vec{U}^{n},\vec{U}^{n-1}),A_2(\vec{U}^{n},\vec{U}^{n-1})$ can be obtained with slight changes from those reported in \cite{fumagalli2024polytopal}.
Specifically, the Navier-Stokes advection term encompasses a modification of the diagonal block corresponding to the fluid velocity $\vec{U}$, while the BJS condition yields the addition of the following matrix in the diagonal blocks corresponding to the solid and fluid velocities $\vec{Z}$ and $\vec{U}$, and their mutual coupling:
\begin{align*}
[G_{\star,\text{\tiny$\triangle$}}]_{ij} &= \sum_{F\in\mathscr F^\Sigma}\int_F \frac{\gamma\mu_\FF}{\sqrt{k_\EE}}(\bm\varphi_\star^j)_\vec{\tau}\cdot(\bm\varphi_\text{\tiny$\triangle$}^i)_\vec{\tau},
\qquad
\star\in\{\PP,\FF\},
\quad
\text{\tiny$\triangle$}\in\{\PP,\FF\},
\end{align*}
where the integrand is expressed in terms of the basis functions of the discrete spaces:
$\spaceW={\rm span}\{\bm\varphi_\PP^i\}_{i=0}^{N_\PP}$, $\spaceV={\rm span}\{{\bm\varphi_\FF^i}\}_{i=0}^{N_\FF}$.

{The resulting discrete scheme \eqref{eq:fullydiscr} is implemented in \textit{FEniCS} - version 2019.1.0 1
\cite{logg2012automated,alnaes2015fenics}, hinging upon the library \textit{multiphenics} (\url{https://github.com/multiphenics/multiphenics}) to deal with the multiphysics nature of the problem.
Although the method can be employed on general polyhedral meshes (as shown in \cref{sec:polyvsstd}), the construction of a three-dimensional polyhedral mesh of a generic domain with internal interfaces is a matter of active research, thus the meshes used to obtain the numerical results of \cref{sec:conv,sec:stokes,sec:navier} are made of tetrahedral elements, generated by means of Gmsh \cite{gmsh}(\url{https://gmsh.info/}).
Nevertheless, to assess the advantages of employing a polyhedral mesh w.r.t.~more classical ones, in an in-house library we implemented a simplified single-physics problem, for which we could generate actually polyhedral meshes: the results of this assessment are reported in \cref{sec:polyvsstd}.}

\section{Verification of convergence estimates}\label{sec:conv}

To verify the theoretical error estimates of \cref{th:conv} and the implemented solver, we report convergence tests on a simplified mesh made of two juxtaposed unit cubes $\Omega_\FF=(0,1)\times(1,0)\times(-1,0), \Omega_\PP=(0,1)^3$, where the interface is $\Sigma=\{\vec{x}=(x,y,z)\colon z=0, (x,y)\in(0,1)^2\}$.
Considering only one fluid compartment in the poroelastic system, namely $J=\{\EE\}$, and setting all physical coefficients of \cref{tab:modelparams} to be equal to 1, except for $\alpha_\EE=0.5$, the following is an exact solution of problem \eqref{eq:interf}-\eqref{eq:NSMPE} for suitable expressions of the source functions $\vec{f}_\PP,g_\EE,\vec{f}_\FF$, the boundary data, 
and the initial conditions $\vec{d}_0,\dot{\vec{d}}_0,p_{\EE0},\vec{u}_0$:
\begin{equation}\label{eq:exactsol}\begin{gathered}
\vec{u}(\vec{x},t) = (2t-t^2)e^{-t}
    \begin{bmatrix} y^Mz^M \\ x^Mz^M \\ \xi \end{bmatrix},
\qquad
\vec{d}(\vec{x},t) = t^2e^{-t}
    \begin{bmatrix} y^Mz^M \\ x^Mz^M \\ -\xi \end{bmatrix},
\\
p(\vec{x},t) = p_\EE(\vec{x},t) = (1-e^{-t})z^M.
\end{gathered}\end{equation}
In particular, we choose $M=5$ and $\xi=1$ and impose Dirichlet conditions for all variables on the whole of $\partial\Omega$.

We simulate 5 time steps with step length $\Delta t=10^{-3}$, chosen small enough to avoid spoiling convergence w.r.t.~space discretization.
In \cref{fig:conv3Dunsteady}, we report the computed errors in the energy norm \eqref{eq:normcoerc}.
The results agree with the convergence order $h^m, m=1,2,3$ predicted by \cref{th:conv}.

\begin{figure}
    \centering
    \begin{tikzpicture}
        \begin{loglogaxis}[width=0.67\textwidth,%
        height=0.5\textwidth,%
        xlabel={$h$},%
        axis line style = thin,%
        font=\scriptsize,%
        legend columns=3,%
        legend style={anchor=south east, at = {(0.95,0.0)},
        draw=none, fill=none}]
            \addplot[color=blue,mark=o] table [x=h, y=errrelP1, col sep=comma]{convStokesPoroelasticityN4N16dt0p001.csv};
            \addlegendentry{$m$=1};
            \addplot[color=red,mark=+] table [x=h, y=errrelP2, col sep=comma]{convStokesPoroelasticityN4N16dt0p001.csv};
            \addlegendentry{$m$=2};
            \addplot[color=olive,mark=square] table [x=h, y=errrelP3, col sep=comma]{convStokesPoroelasticityN4N16dt0p001.csv};
            \addlegendentry{$m$=3};
            \addplot[color=blue,dashed] table [x=h, y=h1, col sep=comma]{convStokesPoroelasticityN4N16dt0p001.csv};
            \addlegendentry{$h^1$};
            \addplot[color=red,dashed] table [x=h, y=h2, col sep=comma]{convStokesPoroelasticityN4N16dt0p001.csv};
            \addlegendentry{$h^2$};
            \addplot[color=olive,dashed] table [x=h, y=h3, col sep=comma]{convStokesPoroelasticityN4N16dt0p001.csv};
            \addlegendentry{$h^3$};
        \end{loglogaxis}
    \end{tikzpicture}
    \caption{Verification test {of \cref{sec:conv}}: computed relative errors in the energy norm \eqref{eq:normcoerc} versus $h$ for different polynomial degrees $m=1,2,3$ (log-log scale).}
    \label{fig:conv3Dunsteady}
\end{figure}

\section{{Computational advantages of employing polyhedral meshes}}\label{sec:polyvsstd}

\begin{figure}
\centering
\includegraphics[width=0.9\textwidth]{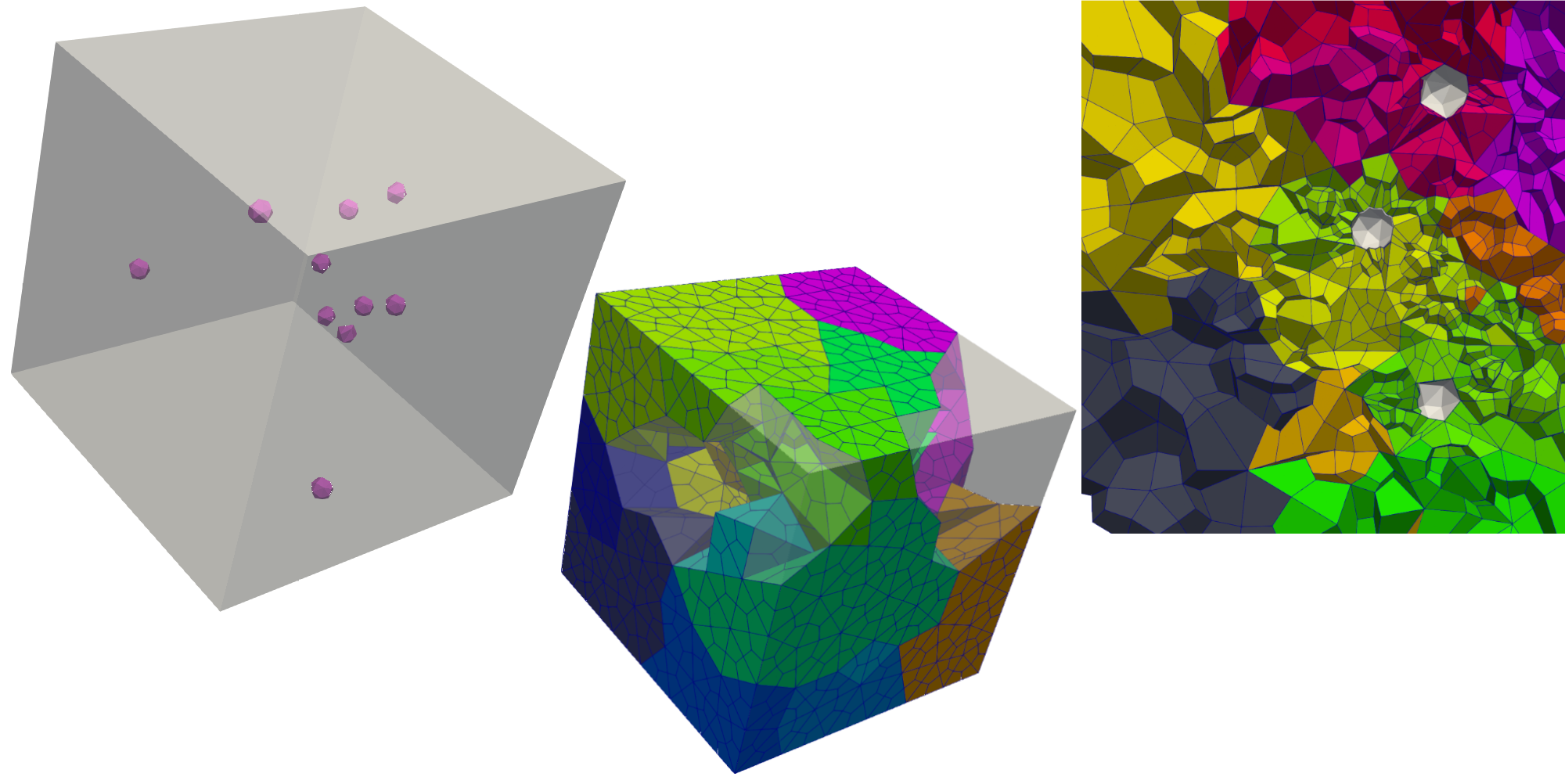}
\caption{{Test case of \cref{sec:polyvsstd}. Left: computational domain with small inclusions (magenta holes). Center: hexahedral mesh $\mathscr T_{26816}$ colored according to the agglomerated polyhedral mesh $\mathscr T_{45}$. Right: clip and zoom of $\mathscr T_{45}$ close to some inclusions (in white).}}
\label{fig:polyvsstd-mesh}
\end{figure}

{
Before applying the PolyDG method to solve the fluid-poromechanics problem of interest, it is important to assess the pros and cons of employing a polytopal mesh instead of a more standard tetrahedral/hexahedral one.
In particular, we focus on the computational cost entailed by the two approaches on a simplified problem that retains some of the characteristics of the application of interest.
Aiming to focus on the specifics of the numerical discretization method, we solve a Poisson problem in this domain, with an exact solution
$
u_\text{ex}(x,y,z)=\sin(\pi x)\sin(\pi y)\sin(\pi z),
$
enforcing Dirichlet boundary conditions on the whole boundary.
To assess the suitability of the methods on a domain with small geometrical details of characteristic size $\epsilon$, we consider the computational domain depicted in \cref{fig:polyvsstd-mesh}, which contains 10 randomly placed inclusions of diameter $\epsilon$.
In this domain, we generate a hexahedral mesh $\mathscr T_{26816}$ of 26816 elements by means of Gmsh \cite{gmsh}.
Over there, we define a tensor-product DG space and quadrature formulas for any finite element polynomial degree $m$.
From $\mathscr T_{26816}$, two polyhedral grids $\mathscr T_{45}, \mathscr T_{90}$ are generated by agglomeration via METIS \cite{metis}: see \cref{fig:polyvsstd-mesh}.

We remark that the hexahedral mesh needs to have a resolution proportional to the inclusion diameter $\epsilon$, at least close to the inclusions.
On the other hand, the size of the polyhedral mesh elements is independent of $\epsilon$, thanks to the possibility of handling elements with many small faces.

The complete DG discretization of the problem is reported in \cref{sec:poisson}: the same numerical scheme is applied to all meshes, and all the simulations discussed in this section were run on a single CPU.

To compare the polyhedral approach to the hexahedral one in terms of both accuracy and computational efficiency, \cref{fig:polyvsstd-results} reports the results of convergence tests w.r.t.\~the polynomial degree $m$: for each of the mesh introduced above, we plot the convergence errors $E_{L^2}=\|u_\text{ex}-u_h\|_{L^2(\Omega)}, E_{H^1}=\|\nabla u_\text{ex}-\nabla_hu_h\|_{L^2(\Omega)}$ and the computational time of the simulations against the number of degrees of freedom (DOFs) entailed by each choice of $m$.

Observing the computed errors, we can notice that the polyhedral approach requires significantly fewer DOFs than the hexahedral one to attain a given level of accuracy.
For example, an error $E_{L^2}$ of less than $\SI{5e-4}{}$ is achieved with either 3780 DOFs (with $m=6$ over $\mathscr T_{45}$) or 5040 DOFs (with $m=5$ over $\mathscr T_{90}$) using an agglomerated polyhedral mesh, while 268160 DOFs are required when using a hexahedral mesh ($m=2$ over $\mathscr T_{26816}$).
Indeed, to capture the geometrical detail of the small inclusions, many hexahedral elements are required in some regions, leading to a large system without gaining in approximation; on the other hand, a single polyhedron with many small faces can preserve the same geometrical (and computational) accuracy.

In terms of computational time, we discuss separately the times required to assemble the system and to solve it, both reported in \cref{fig:polyvsstd-results}.
We notice that solving the significantly smaller systems associated with the polyhedral meshes requires a much smaller computational time than in the hexahedral case.
Regarding the assembly phase, times seem comparable between the polyhedral and hexahedral approaches.
Our current implementation of the 3D solver relies on a quadrature rule based on the sub-tessellation of each polyhedral element into hexahedra. 
Yet, the assembly could be made much more efficient by the quadrature-free strategy proposed in \cite{AntoniettiHoustonPennesi_18}: two-dimensional tests in the open-source library \texttt{lymph} showed more than 20\% reduction in assembly time when using the quadrature-free strategy instead of sub-tessellation \cite{lymph} (\url{https://lymph.bitbucket.io/}), and 3D implementations may yield even further reductions thanks to the recursive nature of this strategy \cite{AntoniettiHoustonPennesi_18}.

To complete this discussion, we point out that agglomeration and setup of the PolyDG method over both $\mathscr T_{45}$ and $\mathscr T_{90}$ required less than 20 seconds in all cases, and the setup for the hexahedral solver was up to 50 seconds.
Moreover, the PolyDG method can also be beneficial in terms of memory requirements, in the case of geometrically detailed domains.
Indeed, the linear system to be assembled is smaller than in the hexahedral case, and it is still sparse with a limited bandwidth. 

The results and discussion presented here can be extended to the more complex case of the fluid-poromechanics problem discussed in the rest of the manuscript:
a realistic brain geometry would contain very fine geometrical details, and since the brain and CSF are not characterized by very large displacements or turbulent flows, the solution of problem \eqref{eq:NSMPE} is relatively smooth.
Moreover, all the conclusions on computational costs would extend to the larger systems solved on a parallel cluster: mesh agglomeration entails a negligible computational effort w.r.t.~assembly or solution of the linear system -- and it needs to be performed only once, even in time-dependent problems -- therefore it can be operated separately on the sub-mesh pertaining to each processor, resulting in negligible inter-process communication overhead.
}

\begin{figure}
\centering
\includegraphics[width=\textwidth]{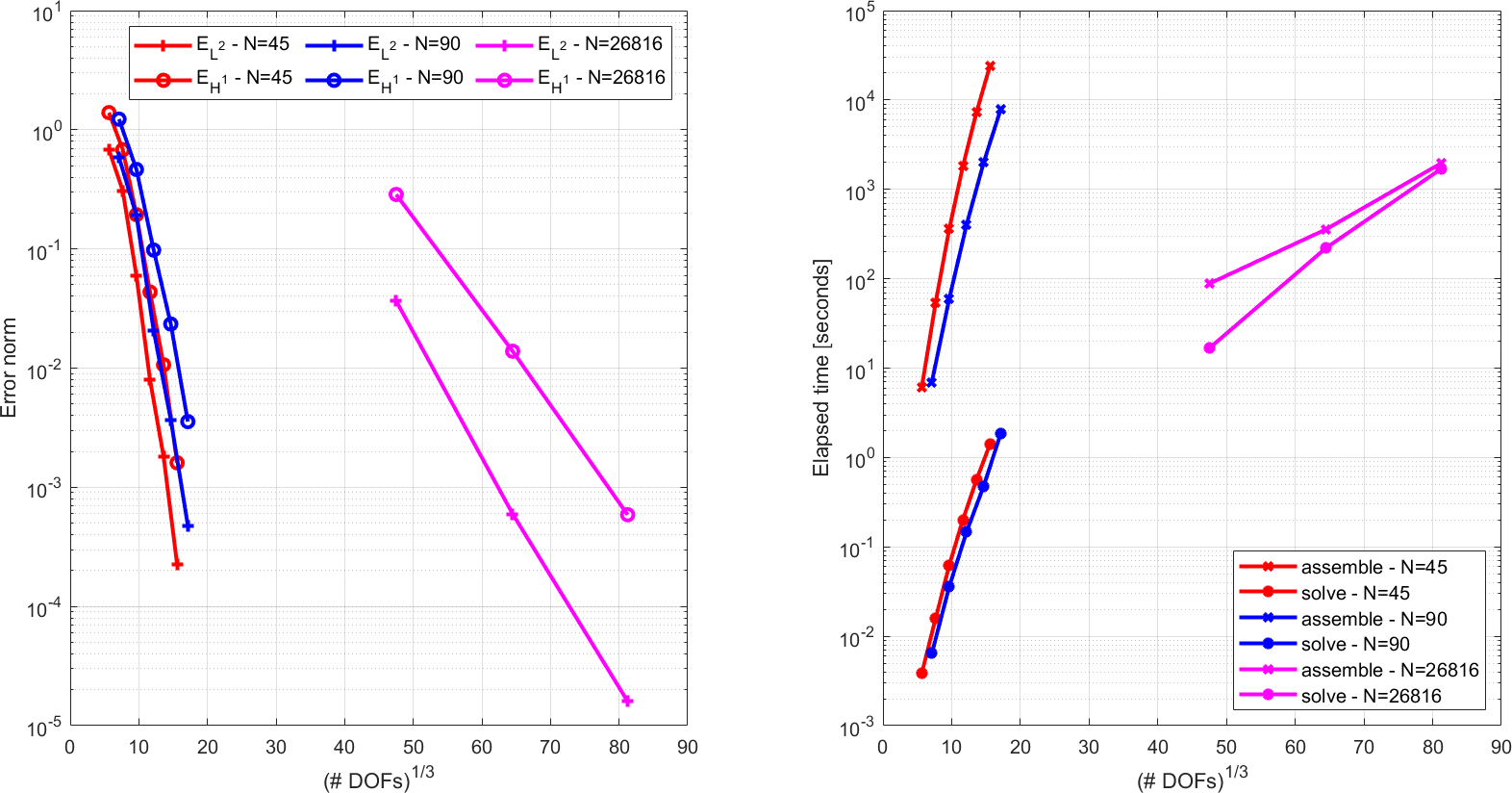}
\caption{{Test case of \cref{sec:polyvsstd}: computational costs using a hexahedral mesh of $N=26816$ elements and two different polyhedral meshes of $N=45$ or $N=91$ elements, for polynomial degrees $m=1,2,3$ (on the hexahedral mesh) and $m=1,2,3,4,5,6$ (on the polyhedral meshes).
Left: convergence errors $E_{L^2}, E_{H^1}$. Right: computational time for assembling the linear system and its solution.}}
\label{fig:polyvsstd-results}
\end{figure}

\section{Application to brain fluid-poromechanics: physiological conditions in idealized geometry}\label{sec:stokes}

In this section, we consider the CSF modeling by the Stokes equations (namely we neglect the nonlinear advection term in \eqref{eq:NSMPE}) and
we apply our multi-physics model to physiological settings, focusing on the CSF compartment.
Specifically, we consider $J=\{\EE\}$ and the values of the physical parameters reported in \cref{tab:modelparams}, with typical physiological values according to
\cite{causemann2022human,holter2017interstitial,tithof2022network}.

\begin{table}
    \centering
    \begin{tabular}{ccl}
        parameter & phys.~values & description\\
        \hline
        $\rho_\PP,\rho_\FF$ &   $\SI{1000}{\kilo\gram\per\cubic\meter}$ & density of the solid tissue and of the CSF\\
        $\mu_\PP$ & $\SI{216}{\pascal}$ & first Lam\'e parameter of the solid \\
        $\lambda$ & $\SI{11567}{\pascal}$ & second Lam\'e parameter of the solid \\
        $\mu_\EE, \mu_\FF$ & $\SI{3.5e-3}{\pascal\second}$ & viscosity of the fluid in compartment $\EE$ and of CSF \\
        $\alpha_\EE$ & $0.49$ & Biot-Willis coefficient of compartment $\EE$ \\
        $c_\EE$ & $\SI{1e-6}{\square\meter\per\newton}$ & storage coefficient of compartment $\EE$ \\
        $\widetilde{k}_\EE$ & $\SI{1e-16}{\square\meter}$ & $k_\EE=\widetilde{k}_\EE I$ permeability tensor for compartment $\EE$ \\
        $\beta_{\EE}^\text{e}$ & $\SI{0}{\square\meter\per\newton\per\second}$ & external coupling coefficient for compartment $\EE\in J$\\
        $\gamma$ & \SI{1}{} & non-dimensional slip rate coefficient at interface $\Sigma$
    \end{tabular}
    \caption{Parameters of model \eqref{eq:NSMPE} with corresponding physiological values \cite{causemann2022human,holter2017interstitial,tithof2022network}.}
    \label{tab:modelparams}
\end{table}

The only non-zero distributed source/sink term is $g_\EE$, 
whereas $\vec{f}_\PP, \vec{f}_\FF$ are set to zero.
This non-zero function is homogenous in space, and its dependence on time is $g_\EE(t)=0.2\pi\sin(2\pi t)$, corresponding to an overall inflow $Q_\text{in}(t)=|\Omega_\PP|g_\EE(t)$ 
which is in the range of the CSF generation rate in physiological conditions: cfr.~\cite{causemann2022human,baledent2001cerebrospinal}.

The idealized geometry displayed in \cref{fig:stokesdp}, left, {although not fully capturing the complex geometry of the actual brain and ventricle system, }retains the same topology of the brain and CSF system:
the porous tissue is contained in $\Omega_\PP$, while the CSF can flow in the ventricle system $\Omega_\FF$, including a duct connecting it to the spinal canal at $\Gamma_\text{out}$.
The volumes of the poroelastic and fluid domains at rest are $|\Omega_\PP|=\SI{9.68}{\milli\liter}$ and $|\Omega_\FF|=\SI{0.89}{\milli\liter}$, respectively.
{and the cylindrical duct has a diameter of $\SI{1}{\centi\meter}$, which is in the physiological range 0.94-$\SI{1.72}{\centi\meter}$ of the spinal canal \cite{ulbrich2014normative}.}
The corresponding mesh is made of
$N=\SI{16834}{}$
elements with an average size
$h=\SI{1.6}{\milli\meter}$.
In terms of boundary conditions, on the outflow $\Gamma_\text{out}$ we prescribe the CSF pressure $\overline{p}^\text{out}=0$ and on the outer wall $\Gamma_\bdout$ we impose no traction on the tissue and no flow of the extracellular CSF.
To ensure that the periodic regime is attained, we simulate 3 periods of the source term $g_\EE$, of duration $T=\SI{1}{\second}$, using a time step $\Delta t=\SI{1e-3}{\second}$.
In the following, all results refer to the third period, with time $t=0$ set at its beginning.

\begin{figure}
\centering
\begin{minipage}[c]{0.25\textwidth}
\includegraphics[width=\textwidth]{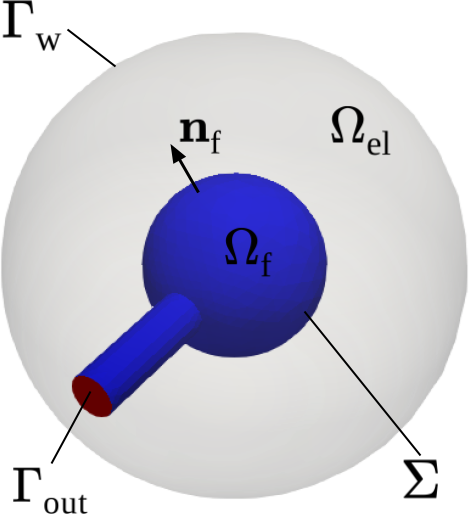}
\end{minipage}
\begin{minipage}[c]{0.74\textwidth}
\includegraphics[width=\textwidth]{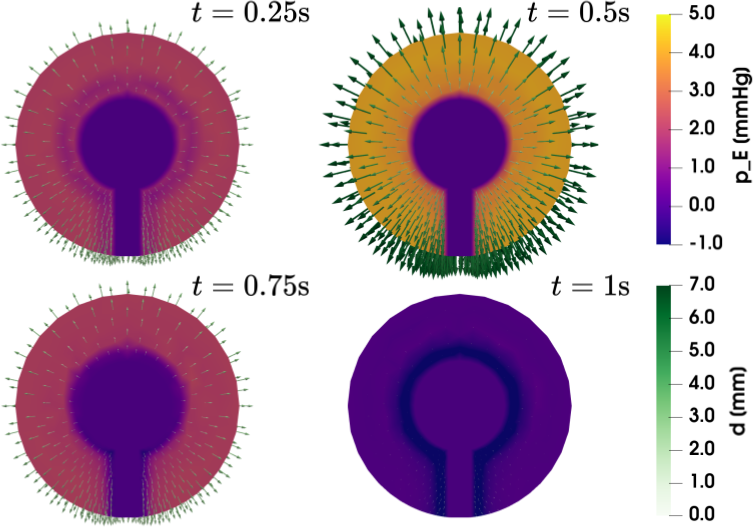}
\end{minipage}
\caption{Test case of \cref{sec:stokes}. Left: computational domain and boundaries. Right: longitudinal clip with computed displacement $\vec{d}_h$ and interstitial pressures $p_{\EE,h}$, for different time snapshots.}
\label{fig:stokesdp}
\end{figure}

\begin{figure}
\centering
\includegraphics[width=0.9\textwidth]{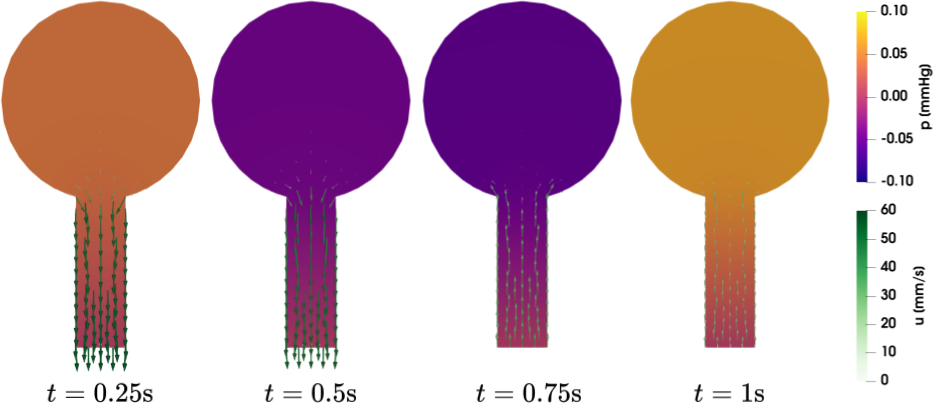}
\caption{Test case of \cref{sec:stokes}: longitudinal clip.
Computed velocity $\vec{u}_h$ and pressures $p_{h}$ in the fluid domain $\Omega_\FF$, for different time snapshots.}
\label{fig:stokesup}
\end{figure}

The displacement $\vec{d}_h$ and the interstitial CSF pressure $p_{\EE,h}$ are reported in \cref{fig:stokesdp}, for selected time snapshots.
The magnitude of $\vec{d}_h$, reaching its peak at $t=\SI{0.5}{\second}$, never exceeds $\SI{6.6}{\milli\meter}$, 
thus it justifies the choice of a linear elastic model for the cerebral tissue, in these settings.
Regarding $p_{\EE,h}$, at the same peak time $t=\SI{0.5}{\second}$ we observe a maximum difference of $\SI{4.4}{\mmHg}$ between the outer boundary $\Gamma_\text{out}$ and the interface $\Sigma$, which is comparable with analogous simulations in brain geometries \cite{causemann2022human,fumagalli2024polytopal}.
We also observe that a slight phase displacement occurs in the late part of the period between $p_{\EE,h}$ and the data $g_\EE$, with a minimum of
$\SI{-1.9}{\mmHg}$ attained at $t=\SI{0.94}{\second}$: this is due to the inertial properties of the system, which are going to be discussed in \cref{sec:navier}.

Regarding the CSF velocity $\vec{u}_h$ and pressure $p_h$ in the Stokes domain $\Omega_\FF$, the results are reported in \cref{fig:stokesup}.
The pressure gradient is significantly smaller than in the poroelastic medium and the velocity magnitude is in the physiological range reported in clinical measurements and computational assessments in the cerebral aqueduct 
\cite{causemann2022human,baledent2001cerebrospinal}.
The positive and increasing pressure gradient observed in the final portion of the period (see~\cref{fig:stokesup}, $t=\SI{1}{\second}$) is due to the aforementioned inertial effects, even more noticeable in the flowrate plots of \cref{fig:flowrates}, where the output flowrate $Q_\text{out}=\int_{\Gamma_\text{out}}\vec{u}_h\cdot\vec{n}\,d\Gamma$ shows 
{a phase delay of 0.11-$\SI{0.12}{\second}$}
w.r.t.~the inflow $Q_\text{in}=\int_{\Omega_\EE}g_\EE\,d\Omega=|\Omega_\EE|g_\EE$ associated to the source term
{, while the period is equal to $\SI{1}{\second}$ for both flowrates.
This final observation holds true in all the cases compared in \cref{sec:navier}.}

{To conclude this section, we point out how the results presented here can vary if a fully detailed brain geometry were considered.
Introducing the multi-chamber ventricle geometry would break the overall symmetry of the domain.
In that case, relatively concentrated pressure gradients would appear in the small ducts connecting the chambers, and the partition of the ventricle interface into multiple regions (instead of the two considered here) could shed light on the heterogeneous distribution of fluid mass exchange between the tissue and the ventricles.
Regarding the pial surface, instead, the computational modeling of the complex gyri and sulci structure and its mechanical interaction with the CSF in the subarachnoid space is still an open issue in the literature, due to the high computational demand involved in this task \cite{sweetman2011three,causemann2022human}.
}

\begin{figure}
    \centering
    \captionsetup[subfigure]{labelformat=empty}
    \setlength{\mywidth}{\textwidth}
    \hspace{-2em}
    \begin{subfigure}{.36\textwidth}
    \begin{tikzpicture}
        \begin{axis}[width=\textwidth,%
        height=0.5\mywidth,%
        xmin=0, xmax=1,
        xlabel={$t[\SI{}{\second}]$},%
        title={$Q_\text{out}\ [\SI{}{\milli\liter\per\second}]$},%
        axis line style = thin,%
        ymajorgrids=true,
        grid style={line width=.1pt, draw=gray!20},
        font=\scriptsize,%
        legend columns=1,%
        legend style={%
        anchor=south west, at = {(0.05,0.05)}}],
            \addplot[color=red,mark=none,
            dashed,
            x filter/.code={\pgfmathparse{\pgfmathresult-2}\pgfmathresult},
            y filter/.code={\pgfmathparse{-\pgfmathresult*1000000}\pgfmathresult}
            ]
            table [x="Time", y="uS:2", col sep=comma]{stokesNoBJS-long_integrOutlet.csv};
            \addlegendentry{Stokes, $\gamma=0$};
            \addplot[color=blue,mark=none,
            dashed,
            x filter/.code={\pgfmathparse{\pgfmathresult-2}\pgfmathresult},
            y filter/.code={\pgfmathparse{-\pgfmathresult*1000000}\pgfmathresult}
            ]
            table [x="Time", y="uS:2", col sep=comma]{stokesBJS-long_integrOutlet.csv};
            \addlegendentry{Stokes, $\gamma=1$};
            \addplot[color=red,mark=none,
            x filter/.code={\pgfmathparse{\pgfmathresult-2}\pgfmathresult},
            y filter/.code={\pgfmathparse{-\pgfmathresult*1000000}\pgfmathresult}
            ]
            table [x="Time", y="uS:2", col sep=comma]{NSnoBJS-long_integrOutlet.csv};
            \addlegendentry{N-S, $\gamma=0$};
            \addplot[color=blue,mark=none,
            x filter/.code={\pgfmathparse{\pgfmathresult-2}\pgfmathresult},
            y filter/.code={\pgfmathparse{-\pgfmathresult*1000000}\pgfmathresult}
            ]
            table [x="Time", y="uS:2", col sep=comma]{NSBJSbwStab-long_integrOutlet.csv};
            \addlegendentry{N-S, $\gamma=1$};
\addplot+[domain=0:1,samples=201, trig format=rad, color=black, mark=none, dotted] {9.6845*0.2*pi*sin(2*pi*\x)};
\addlegendentry{$Q_\text{in}$};
    \end{axis}
    \end{tikzpicture}
    \end{subfigure}
    \hspace{-2em}
    \begin{subfigure}{.36\textwidth}
    \begin{tikzpicture}
        \begin{axis}[width=\textwidth,%
        height=0.5\mywidth,%
        xmin=0, xmax=1,
        xlabel={$t[\SI{}{\second}]$},%
        title={$Q_\Sigma\ [\SI{}{\milli\liter\per\second}]$},%
        axis line style = thin,%
        ymajorgrids=true,
        grid style={line width=.1pt, draw=gray!20},
        font=\scriptsize,%
        legend columns=1,%
                legend style={%
        anchor=south west, at = {(0.05,0.05)}}]
    \addplot+[color=red,mark=none,
            dashed,
            x filter/.code={\pgfmathparse{\pgfmathresult-2}\pgfmathresult},
            y filter/.code={\pgfmathparse{-\pgfmathresult*1000000}\pgfmathresult}
            ]
            table [x="Time", y="interfU", col sep=comma]
            {stokesNoBJS-long_integrInterf.csv};
            \addlegendentry{Stokes, $\gamma=0$};
            \addplot[color=blue,mark=none,
            dashed,
            x filter/.code={\pgfmathparse{\pgfmathresult-2}\pgfmathresult},
            y filter/.code={\pgfmathparse{-\pgfmathresult*1000000}\pgfmathresult}
            ]
            table [x="Time", y="interfU", col sep=comma]
            {stokesBJS-long_integrInterf.csv};
            \addlegendentry{Stokes, $\gamma=1$};
            \addplot[color=red,mark=none,
            x filter/.code={\pgfmathparse{\pgfmathresult-2}\pgfmathresult},
            y filter/.code={\pgfmathparse{-\pgfmathresult*1000000}\pgfmathresult}
            ]
            table [x="Time", y="interfU", col sep=comma]
            {NSnoBJS-long_integrInterf.csv};
            \addlegendentry{N-S, $\gamma=0$};
            \addplot[color=blue,mark=none,
            x filter/.code={\pgfmathparse{\pgfmathresult-2}\pgfmathresult},
            y filter/.code={\pgfmathparse{-\pgfmathresult*1000000}\pgfmathresult}
            ]
            table [x="Time", y="interfU", col sep=comma]
            {NSBJSbwStab-long_integrInterf.csv};
            \addlegendentry{N-S, $\gamma=1$};
\addplot[domain=0:1,samples=201, trig format=rad, color=black, mark=none, dotted]{9.6845*0.2*pi*sin(2*pi*x)};
\addlegendentry{$Q_\text{in}$};
        \end{axis}
    \end{tikzpicture}
    \end{subfigure}
    \hspace{-2em}
    \begin{subfigure}{.36\textwidth}
    \begin{tikzpicture}
        \begin{axis}[width=\textwidth,%
        height=0.5\mywidth,%
        xmin=0, xmax=1,
        xlabel={$t[\SI{}{\second}]$},%
        title={$Q_{\Sigma_\text{can}}\ [\SI{}{\milli\liter\per\second}]$},%
        axis line style = thin,%
        ymajorgrids=true,
        grid style={line width=.1pt, draw=gray!20},
        font=\scriptsize,%
        legend columns=1,%
                legend style={%
        anchor=south west, at = {(0.05,0.05)}}]
    \addplot+[color=red,mark=none,
            dashed,
            x filter/.code={\pgfmathparse{\pgfmathresult-2}\pgfmathresult},
            y filter/.code={\pgfmathparse{-\pgfmathresult*1000000}\pgfmathresult}
            ]
            table [x="Time", y="interfU", col sep=comma]
            {stokesNoBJS-long_integrActualCan.csv};
            \addlegendentry{Stokes, $\gamma=0$};
            \addplot[color=blue,mark=none,
            dashed,
            x filter/.code={\pgfmathparse{\pgfmathresult-2}\pgfmathresult},
            y filter/.code={\pgfmathparse{-\pgfmathresult*1000000}\pgfmathresult}
            ]
            table [x="Time", y="interfU", col sep=comma]
            {stokesBJS-long_integrActualCan.csv};
            \addlegendentry{Stokes, $\gamma=1$};
            \addplot[color=red,mark=none,
            x filter/.code={\pgfmathparse{\pgfmathresult-2}\pgfmathresult},
            y filter/.code={\pgfmathparse{-\pgfmathresult*1000000}\pgfmathresult}
            ]
            table [x="Time", y="interfU", col sep=comma]
            {NSnoBJS-long_integrActualCan.csv};
            \addlegendentry{N-S, $\gamma=0$};
            \addplot[color=blue,mark=none,
            x filter/.code={\pgfmathparse{\pgfmathresult-2}\pgfmathresult},
            y filter/.code={\pgfmathparse{-\pgfmathresult*1000000}\pgfmathresult}
            ]
            table [x="Time", y="interfU", col sep=comma]
            {NSBJSbwStab-long_integrActualCan.csv};
            \addlegendentry{N-S, $\gamma=1$};
        \end{axis}
    \end{tikzpicture}
    \end{subfigure}
\caption{Test cases of \cref{sec:stokes,sec:navier}.
    Computed flowrates over a period, for different modeling choices (N-S: Navier-Stokes). Left: outlet flowrate $Q_\text{out}=\int_{\Gamma_\text{out}}\vec{u}_h\cdot\vec{n}$; {center}: interface flowrate $Q_\Sigma=\int_{\Sigma}\vec{u}_h\cdot\vec{n}_\PP$
    {; right: \emph{canal} interface flowrate $Q_{\Sigma_\text {can}} = \int_{\Sigma_\text{can}}\vec{u}_h\cdot\vec{n}_\PP$}.
    The dashed line corresponds to the distributed CSF inflow flowrate $Q_\text{in}=\int_{\Omega_\PP}g_\EE$, for reference.
    }
    \label{fig:flowrates}
\end{figure}

\section{Application to brain poro-mechanics: effects of BJS conditions and advection term}\label{sec:navier}

This section aims to assess the effects of the modeling choices for the advection term in the fluid part of system \eqref{eq:NSMPE} -- namely either Stokes or Navier-Stokes equations -- and of the BJS condition \eqref{eq:BCtgstress} at the interface -- namely either $\gamma=0$ or $\gamma=1$.
Since a semi-implicit treatment of the advection term is employed, as reported in \eqref{eq:linearization}, the problem to be solved at each time step is linear also in the Navier-Stokes case.
Moreover, backflow stabilization is needed in the Navier-Stokes case for the latest portion of the period:
the term $\int_{\Gamma_\text{out}}\frac{\rho_\FF}{2}\min\left\{0,\vec{u}_h^n\cdot\vec{n}\right\}\vec{u}^{n+1}_h\cdot\vec{v}_h$ is added to the left-hand side of the fully discrete problem \eqref{eq:fullydiscr} \cite{backflowstab}.

\subsection{Effects of fluid inertia: Stokes vs.~Navier-Stokes}

Observing the outlet flowrate $Q_\text{out}$ and the interface flowrate $Q_\Sigma=\int_\Sigma\vec{u}_h\cdot\vec{n}_\PP$ displayed in \cref{fig:flowrates}, we can notice that the differences are negligible between the Stokes and Navier-Stokes case (for the same value of $\gamma$).
Indeed, the outlet Reynolds number $\text{Re}=\frac{\rho_\FF Q_\text{out}}{\mu_\FF \pi D}$, $D$ being the diameter of the circular outlet section $\Gamma_\text{out}$, is in the range $(40,50)$ for all cases.

{However, some differences can be observed when focusing on the cylindrical duct.
Indeed, denoting by $\Sigma_\text{can}$ the lateral surface of this canal, the flowrate $Q_{\Sigma_\text{can}}=\int_{\Sigma_\text{can}}\vec{u}_h\cdot\vec{n}_\PP$ reported in \cref{fig:flowrates} displays significant amplitude differences and a phase difference (either positive or negative) of $\SI{0.06}{\second}$ between the Stokes and Navier-Stokes cases.}
{The same} phase difference between the Stokes and Navier-Stokes model 
{can also be appreciated}
in the interface-averaged pressure $P_\Sigma=\frac{1}{|\Sigma|}\int_\Sigma p_h$ reported in \cref{fig:pressures}, particularly in the case $\gamma=1$ corresponding to BJS conditions: this difference can be ascribed to the additional inertia accounted for by the advection term in Navier-Stokes equations.
{The very same observation also holds if $P_{\Sigma_\text{can}}=\frac{1}{|\Sigma_\text{can}|}\int_{\Sigma_\text{can}} p_h$ is considered instead of $P_\Sigma$.}

Overall, the results discussed here quantitatively show that the choice of the simpler Stokes model is 
enough to represent the {general features of the }flow in the regime of interest. 
{The differences observed when focusing on the canal interface $\Sigma_\text{can}$ do not have a significant effect on the rest of the fluid domain, and these differences would be even less pronounced in the aqueduct of Sylvius of an actual brain geometry, which is characterized by an even smaller diameter ($\sim1$-$\SI{3}{\milli\meter}$) and thus less inertial effects.}

\begin{figure}
    \centering
    \setlength{\mywidth}{\textwidth}
    \begin{tikzpicture}
        \begin{axis}[width=0.5\textwidth,%
        height=0.5\mywidth,%
        xmin=0, xmax=1,
        xlabel={$t[\SI{}{\second}]$},%
        ylabel={$P_\Sigma\ [\SI{}{\mmHg}]$},%
        axis line style = thin,%
        ymajorgrids=true,
        grid style={line width=.1pt, draw=gray!20},
        font=\scriptsize,%
        legend columns=1,%
        legend style={%
        anchor=south west, at = {(0.05,0.04)}}],
            \addplot[color=red,mark=none,
            dashed,
            x filter/.code={\pgfmathparse{\pgfmathresult-2}\pgfmathresult},
            ]
            table [x="Time", y="pSigma", col sep=comma]{stokesNoBJS-long_pSigma.csv};
            \addlegendentry{Stokes, $\gamma=0$};
            \addplot[color=blue,mark=none,
            dashed,
            x filter/.code={\pgfmathparse{\pgfmathresult-2}\pgfmathresult},
            ]
            table [x="Time", y="pSigma", col sep=comma]{stokesBJS-long_pSigma.csv};
            \addlegendentry{Stokes, $\gamma=1$};
            \addplot[color=red,mark=none,
            x filter/.code={\pgfmathparse{\pgfmathresult-2}\pgfmathresult},
            ]
            table [x="Time", y="pSigma", col sep=comma]{NSnoBJS-long_pSigma.csv};
            \addlegendentry{N-S, $\gamma=0$};
            \addplot[color=blue,mark=none,
            x filter/.code={\pgfmathparse{\pgfmathresult-2}\pgfmathresult},
            ]
            table [x="Time", y="pSigma", col sep=comma]{NSBJSbwStab-long_pSigma.csv};
            \addlegendentry{N-S, $\gamma=1$};
    \end{axis}
    \end{tikzpicture}
\caption{Test cases of \cref{sec:stokes,sec:navier}.
    Interface-average fluid pressure $P_\Sigma=\frac{1}{|\Sigma|}\int_\Sigma p_h$ over a period for different modeling choices (N-S: Navier-Stokes).}
    \label{fig:pressures}
\end{figure}

\begin{figure}
\centering
\includegraphics[width=\textwidth]{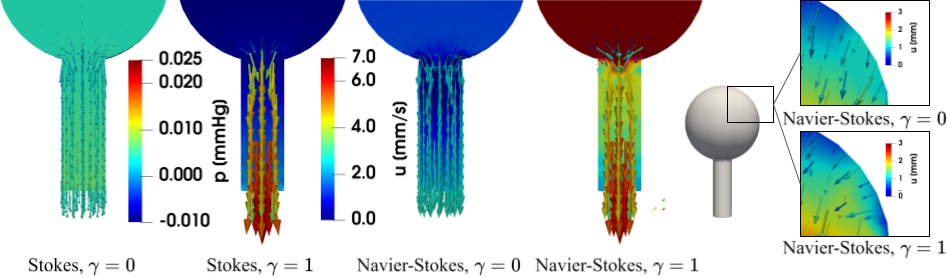}
\caption{Test cases of \cref{sec:stokes,sec:navier}: longitudinal clip.
Main panels: computed velocity $\vec{u}_h$ and pressures $p_{h}$ in the fluid domain $\Omega_\FF$ at $t=\SI{0.24}{\second}$, for different modeling choices.
Right: zoom on $\vec{u}_h$ distribution in the boxed region.}
\label{fig:stokesupALLt0p24}
\end{figure}

\subsection{Effects of BJS conditions}

To assess the effects of the BJS condition \eqref{eq:BCtgstress}, we can compare the cases $\gamma=0$ (no tangential stress at the pores) to the cases $\gamma=1$ (BJS condition) in \cref{fig:flowrates,fig:pressures}.
The approximately sinusoidal evolution of flowrates and pressure exhibits a smaller amplitude in the case $\gamma=0$, particularly significant in the interface quantities $P_\Sigma$ and $Q_\Sigma${ and even more in the \emph{canal} interface flowrate $Q_{\Sigma_\text{can}}$}.
An interpretation of this effect is that the friction in the tangential direction introduced by the BJS interface condition makes CSF velocity orient in the direction orthogonal to the interface, thus increasing the amount of CSF flowing through it at each given time.
Due to the incompressibility of the fluid model, this increased flow corresponds to an increased pressure difference $P_\Sigma-\overline{p}_\text{out} = P_\Sigma$.
This is confirmed by the results reported in \cref{fig:stokesupALLt0p24}: when the BJS condition is applied, $\vec{u}_h$ is mostly normal to the interface in the upper region of $\Omega_\FF$ (see the boxed panels in \cref{fig:stokesupALLt0p24}), and an increased flowrate is observable in the vertical duct. 
Moreover, the friction introduced by the BJS condition yields a bell-shaped profile in the {canal}, in contrast with the flat profile of the case $\gamma=0$.
Such a profile better represents the actual CSF flow in the cerebral aqueduct {and spinal canal} \cite{sweetman2011three}.
To conclude this discussion, it is interesting to notice that, despite the significant differences in the pressure and velocity distribution, the outlet flowrate $Q_\text{out}$ shows little discrepancies between the different cases, since the outflow is mostly driven by conservation of mass.

\section{Conclusions}

Motivated by the modeling of cerebrospinal fluid (CSF) flow in the brain, in this work we studied a coupled multi-domain system encompassing multiple-network poroelasticity and (Navier-)Stokes equations, with a particular focus on the interface conditions between the two physical domains.
The polytopal discontinuous Galerkin method for space discretization introduced in \cite{fumagalli2024polytopal} was extended to consider the nonlinear advection term of Navier-Stokes equations in the fluid domain and Beavers-Joseph-Saffman (BJS) interface conditions: an a priori analysis of the resulting method in the Stokes case proved it to be stable and optimally convergent.
The method was numerically verified by means of convergence tests in three dimensions{, and a quantitative comparison with a standard DG approach on a hexahedral mesh showed the advantages of the polyhedral approach.}.
Considering an idealized geometry, representative of the main topological characteristics of the brain tissue and ventricular system, the method was employed to represent porous tissue dynamics and CSF flow in physiological settings, obtaining results in partial consistency with the literature.
Analyzing the effects of the advection term in the fluid model, it was quantitatively observed that neglecting the advection term in the Navier-Stokes equations does not have a significant impact on the CSF velocity distribution (at least in physiological conditions), though it may yield a slightly inaccurate prediction of the pressure distribution.
Finally, it was highlighted how the BJS condition strongly affects the CSF pressure and velocity distribution: the results obtained with the enforcement of this condition better represent the CSF flow in the brain, especially in terms of its profile in the {canal}.

To further develop the computational model proposed in this work, {a crucial improvement would be to consider patient-specific geometries, including a detailed representation of the ventricle chambers, the interventricular foramina and aqueduct of Sylvius, and the folds of the brain cortex.
This would allow to obtain physiologically accurate results, which are necessary for a comparison with clinical measurements, thus opening the way to data-driven and possibly personalized calibration of the model parameters.}
{To build up a computational domain representing the complex brain geometry}, image segmentation pipelines developed in the literature 
(see, e.g., \cite{mardal2022mathematical,fumagalli2020image})
should be combined with mesh generation and agglomeration algorithms for the construction of a polyhedral mesh 
\cite{dassi2022virtual,feder2024r3mg,antonietti2024agglomeration}.

{Such a complex model would require further developments also in the computational solver, to make the numerical simulations affordable.
The quantitative comparison between the PolyDG method and a standard DG method on a hexahedral mesh reported here
provided a preliminary -- yet indicative -- assessment of the advantages of the polyhedral approach in handling fine geometric details, 
in terms of accuracy, memory requirements, and efficiency in the solution of the problem.
Yet, to actually translate the method to a realistic geometry, some directions of further development must be followed.
First, the assembly time of the PolyDG discretization, currently comparable with the standard DG one, could be further significantly improved by a quadrature-free strategy \cite{AntoniettiHoustonPennesi_18,lymph}.
Then, a parallel solver could be attained with
a suitable integration of the above-mentioned agglomeration algorithms, either as a pre-processing stage or by having each parallel process perform agglomeration on its own sub-mesh \cite{feder2024r3mg,pseudostress}.
Finally, scalable preconditioners for multi-physics problems could be introduced, hinging upon either operator-based approaches or inexact factorization \cite{mardal2004uniform,boon2021robust, both2022iterative, deparis2014parallel,ferronato2019general}.
}

{In terms of model personalization}, the multiple-network modeling of tissue perfusion could be exploited to assimilate clinical measurements of the blood flow in the cerebral vasculature directly into the model, without the need for a pre-processing phase to estimate CSF generation from such data, or to combine the system with more detailed computational models of the heart pulsation
\cite{quarteroni2017cardiovascular,hirschhorn2020fluid,bucelli2023mathematical}.
Finally, an extension to more complex hyperelastic rheologies of the brain tissue should be envisaged, to more accurately represent its mechanical response
\cite{budday2017rheological}.

\section*{Acknowledgments}
The author has been supported by \emph{ICSC--Centro Nazionale di Ricerca in High Performance Computing, Big Data, and Quantum Computing} funded by the European Union--NextGenerationEU.
The present research is part of the activities of \lq\lq Dipartimento di Eccellenza 2023-2027'', Dipartimento di Matematica, Politecnico di Milano.
The author is member of GNCS-INdAM and acknowledges the support of the GNCS project CUP E53C23001670001.

\appendix

\section{Bilinear forms and functionals of the continuous and semi-discrete problem}\label{sec:forms}
{
The forms appearing in the continuous weak problem \cref{eq:weak} are defined as follows:
\begin{equation}\label{eq:formsContPb}\begin{aligned}
a_\PP:\vec{W}\times\vec{W}\to\mathbb R,
&\quad a_\PP(\vec{d},\vec{w}) = (\sigma_\PP(\vec{d}),\varepsilon(\vec{w}))_{\Omega_\PP},\\
a_\jj:Q_J\times Q_J\to\mathbb R,
&\quad a_\jj(p_\jj,q_\jj) = \left(\frac{1}{\mu_\jj}k_\jj\nabla p_\jj,\nabla q_\jj\right)_{\Omega_\PP} ,\\
C_\jj:[Q_J]^{N_J}\times Q_J\to\mathbb R,
&\quad
C_\jj(\{p_\kk\}_{\text{\tiny$\kk\in J$}},q_\jj) = 
\sum_{\kk\in J}(\beta_{\kk\jj}(p_\jj-p_\kk), q_\jj)_{\Omega_\PP}
+ (\beta_{\jj}^\text{e}p_\jj, q_\jj)_{\Omega_\PP} ,\\
\\
a_\FF:\vec{V}\times\vec{V}\to\mathbb R,
&\quad a_\FF(\vec{u},\vec{v}) = (\tau_\FF(\vec{u}),\varepsilon(\vec{v}))_{\Omega_\FF},\\
N_\FF:\vec{V}\times\vec{V}\times\vec{V}\to\mathbb R,
&\quad N_\FF(\vec{u}',\vec{u},\vec{v}) = \left(\rho_\FF(\vec{u}'\cdot\nabla)\vec{u}+\frac{\rho_\FF}{2}(\nabla\cdot\vec{u}')\vec{u},\vec{v}\right)_{\Omega_\FF},\\
b_\jj:Q_J\times\vec{W}\to\mathbb R,
&\quad b_\jj(q_\jj,\vec{w}) = -(\alpha_\jj q_\jj,\Div\vec{w})_{\Omega_\PP} ,\\
b_\FF:Q\times\vec{V}\to\mathbb R,
&\quad b_\FF(q,\vec{v}) = -(q,\Div\vec{v})_{\Omega_\FF},\\
F_\PP:\vec{W}\to\mathbb R,
&\quad F_\PP(\vec{w}) = (\vec{f}_\PP,\vec{w})_{\Omega_\PP},\\
F_\jj:Q_J\to\mathbb R,
&\quad F_\jj(q_\jj) = (g_\jj,q_\jj)_{\Omega_\PP} ,\\
F_\FF:\vec{V}\to\mathbb R,
&\quad F_\FF(\vec{v}) = (\vec{f}_\FF,\vec{v})_{\Omega_\FF},\\
\mathfrak J:Q_J\times\vec{W}\times\vec{V}\to\mathbb R,
&\quad \mathfrak J(p_\EE,\vec{w},\vec{v}) = \int_\Sigma p_\EE\left(\vec{w}\cdot\vec{n}_\PP+\vec{v}\cdot\vec{n}_\FF\right)d\Sigma,\\
\mathfrak G:(\vec{V}\oplus\vec{W})\times(\vec{V}\oplus\vec{W})&\to\mathbb R,
\quad \mathfrak G(\vec{z}_1, \vec{z}_2) = \int_\Sigma \frac{\gamma\mu_\FF}{\sqrt{k_\EE}}\left(\vec{z}_1\right)_\vec{\tau}\cdot\left(\vec{z}_2\right)_\vec{\tau}d\Sigma.
\end{aligned}\end{equation}
}

The forms building up the semidiscrete problem \eqref{eq:DG} are defined as follows:
\begin{subequations}\label{eq:formsAllTogether}\begin{align}
    \mathcal L_\PP(\vec{d},\{p_\kk\}_{\kk\in J}; \vec{w}) &= \mathcal A_\PP(\vec{d},\vec{w}) + \sum_{\kk\in J}\mathcal B_\kk(p_\kk,\vec{w}) 
    , \\
    \mathcal L_\jj(\{p_\kk\}_{\kk\in J},\partial_t\vec{d};q_\jj) &=  \mathcal A_\jj(p_\jj,q_\jj) + \mathcal C_\jj(\{p_\kk\}_{\kk\in J},q_\jj) - \mathcal B_\jj(q_\jj,\partial_t \vec{d}) 
    , \qquad \forall\jj\in J, \\
    \mathcal L_\FF(\vec{u},p;\vec{v},q) &= \mathcal{N}_\FF(\vec{u},\vec{u},\vec{v}) + \mathcal A_\FF(\vec{u},\vec{v}) + \mathcal B_\FF(p,\vec{v}) 
    - \mathcal B_\FF(q,\vec{u}) + \mathcal S(p,q),\label{eq:formsAllTogetherFF}
\end{align}\end{subequations}
where
\begin{subequations}\label{eq:formsSeparated1}\begin{align}
    \begin{split}
    \mathcal A_\PP(\vec{d},\vec{w}) &= \int_{\Omega_\PP}\sigma_\PP(\vec{d})\colon\varepsilon_h(\vec{w})
    \\&\qquad
    - \sum_{F\in\mathscr F_\PP^\II\cup\mathscr F_\PP^\DD}
    \int_F\left(\average{\sigma_\PP(\vec{d})}\colon\jump{\vec{w}}+\jump{\vec{d}}\colon\average{\sigma_\PP(\vec{w})}
    -
    \eta\jump{\vec{d}}\colon\jump{\vec{w}}\right),
    \end{split}\\
    \mathcal F_\PP(\vec{w}) &= \int_{\Omega_\PP}\vec{f}_\PP\cdot\vec{w}
    ,\\
    \mathcal B_\jj(p_\jj,\vec{w}) &= -\int_{\Omega_\PP}\alpha_\jj p_\jj\,\Div_h\,\vec{w}
    + \sum_{F\in\mathscr F_\PP^\II\cup\mathscr F_\PP^{\DD_\jj}}
    \int_F\alpha_\jj \average{p_\jj I}\colon\jump{\vec{w}}, 
    ,\\
    \begin{split}\mathcal A_\jj(p_\jj,q_\jj) &= \int_{\Omega_\PP}\mu_\jj^{-1}k_\jj\nabla_h p_\jj\cdot\nabla_h q_\jj
    \\&\hspace{-1em}- \sum_{F\in\mathscr F_\PP^\II\cup\mathscr F_\PP^{\DD_\jj}}
    \int_F\left(\average{\mu_\jj^{-1}k_\jj\nabla_h p_\jj }\cdot\jump{q_\jj } +\jump{p_\jj }\cdot\average{\mu_\jj^{-1}k_\jj\nabla_h q_\jj }
    -
    \zeta_\jj\jump{p_\jj}\cdot\jump{q_\jj}
    \right)
    ,\end{split}\\
    \mathcal C_\jj(\{p_\kk\}_{\kk\in J},q_\jj) &= \int_{\Omega_\PP}\sum_{\kk\in J}\beta_{\kk\jj}(p_\jj-p_\kk)q_\jj + \int_{\Omega_\PP}\beta_\jj^\text{e}p_\jj q_\jj
    ,\\
    \mathcal F_\jj(q_\jj) &= \int_{\Omega_\PP}g_\jj q_\jj 
    ,
\end{align}\end{subequations}
\begin{subequations}\label{eq:formsSeparated2}\begin{align}
    \begin{split}\mathcal A_\FF(\vec{u},\vec{v}) &= \int_{\Omega_\FF}\tau_\FF(\vec{u})\colon\varepsilon_h(\vec{v})
    \\&- \sum_{F\in\mathscr F_\FF^\II\cup\mathscr F_\FF^\DD}
    \int_F\left(\average{\tau_\FF(\vec{u})}\colon\jump{\vec{v}}+\jump{\vec{u}}\colon\average{\tau_\FF(\vec{v})}
    -
    \gamma_\vec{v}\jump{\vec{u}}\colon\jump{\vec{v}}
    \right)
    ,\end{split}\\
    \begin{split}\mathcal N_\FF(\vec{u}',\vec{u},\vec{v}) &= \int_{\Omega_\FF}\left(\rho_\FF(\vec{u}'\cdot\nabla_h)\vec{u}\cdot\vec{v} + \frac{\rho_\FF}{2}(\nabla_h\cdot\vec{u}')\vec{u}\cdot\vec{v}\right)
    \\&- \sum_{F\in\mathscr F_\FF^\II}
    \int_F\left(\rho_\FF\jump{\vec{u}}\colon\average{\vec{u}'}\otimes\average{\vec{v}} + \frac{\rho_\FF}{2}\jump{\vec{u}'}\colon I\average{\vec{u}\cdot\vec{v}}\right)
    ,\end{split}\\
    \mathcal B_\FF(p,\vec{v}) &= -\int_{\Omega_\FF} p\,\Div_h\,\vec{v}
    + \sum_{F\in\mathscr F_\FF^\II\cup\mathscr F_\FF^{\DD}}
    \int_F \average{p I}\colon\jump{\vec{v}}
    ,\\
    \mathcal F_\FF(\vec{v}) &= \int_{\Omega_\FF}\vec{f}_\FF\cdot\vec{v}
    ,\\
    \mathcal S(p,q) &=
    \sum_{F\in\mathscr F_\FF^\II}
    \int_F\gamma_p\jump{p}\cdot\jump{q}
    ,\\
    \mathcal J(p_\EE, \vec{w},\vec{v}) &=
        \sum_{F\in\mathscr{F}^\Sigma}\int_F\left(
    \average{p_\EE I} \colon \jump{\vec{w},\vec{v}}
    \right)
    ,\\
    \mathcal G(\vec{v}_1-\vec{w}_1,\vec{v}_2-\vec{w}_2) &=
        \sum_{F\in\mathscr{F}^\Sigma}\int_F
    \frac{\gamma\mu_\FF}{\sqrt{k_\EE}}\jump{\vec{w}_1,\vec{v}_1}_\vec{\tau} \cdot \jump{\vec{w}_2,\vec{v}_2}_\vec{\tau},
\end{align}\end{subequations}
where $\nabla_h, \varepsilon_h, \Div_h$ denote the element-wise gradient, symmetric gradient, and divergence operators, respectively, and the stress tensors $\sigma_\PP,\tau_\FF$ are implicitly defined in terms of these piecewise operators.
The parameters $\eta,\zeta_\jj,\gamma_\vec{v},\gamma_p$ appearing in these forms are defined as follows \cite{corti2022numerical,AMVZ22}:
\begin{equation}\label{eq:penaltyparams}
\eta = \overline{\eta}\frac{\overline{\mathbb C}_\PP^K}{\{h\}_\text{H}},
\qquad
\zeta_\jj = \overline{\zeta}_\jj\frac{\overline{k}_\jj^K}{\sqrt{\mu_\jj}\{h\}_\text{H}},
\qquad
\gamma_\vec{v} = \overline{\gamma}_\vec{v}\frac{\mu}{\{h\}_\text{H}},
\qquad
\gamma_p = \overline{\gamma}_p\{h\}_\text{H},
\end{equation}
where
$\{h\}_\text{H}$ denotes the harmonic average on $K^{\pm}$ (with $\{h\}_\text{H}=h_K$ on Dirichlet faces),
$\overline{\mathbb C}_\PP^K = \|\mathbb C_\PP^{1/2}|_K\|_2^2$ and $\overline{k}_\jj^K = \|k_\jj^{1/2}|_K\|_2^2$ are the $L^2$-norms of the symmetric second-order tensors appearing in the elasticity and Darcy equations, for each $K\in\mathscr T_{h,\PP}$,
and
$\overline{\eta},\overline{\zeta_\jj}\ \forall\jj\in J,\overline{\gamma}_\vec{v},\overline{\gamma}_p$ are penalty constants to be chosen large enough.
In all the numerical tests of the present work, all these penalty constants are set to 10.

\section{Stein interpolation results}\label{sec:stein}

If $\Omega$ is a Lipschitz-regular domain,
the following interpolation results hold for the Stein extension operator $\mathscr E_K: H^s(\Omega)\to H^s(\mathbb R^d)$
(cf.~\cite{fumagalli2024polytopal,perTraceInterp}):
\begin{lemma}\label{th:interp}
Under
\cref{hp},
the following estimates hold:
\[\begin{aligned}
    &\forall (\vec{w},\{q_\jj\}_{\jj\in J},\vec{v},q)\in [H^{m+1}(\mathscr T_{h,\PP})]^{d+N_J}\times[H^{m+1}(\mathscr T_{h,\FF})]^{d+1}\\
    &\exists (\vec{w}_I,\{q_{\jj I}\}_{\jj\in J},\vec{v}_I,q_I)\in\spaceW\times
    [\spaceQj]^{N_J}
    \times\spaceV\times\spaceQ \quad \text{such that}
    \\
    &i)\ \normcont{(\vec{w}-\vec{w}_I,\{q_\jj-q_{\jj I}\}_{\jj\in J},\vec{v}-\vec{v}_I,q-q_I)}^2 \\
    &\qquad\lesssim \sum_{K\in\mathscr T_{h,\PP}}h_K^{2m}\left(\|\mathcal E_K\vec{w}\|_{[H^{m+1}(\widehat{K})]^d}^2+\sum_{\jj\in J}\|\mathcal E_\jj q_\jj\|_{H^{m+1}(\widehat{K})}^2+\|\mathcal E_K\vec{d}\|_{[H^{m+1}(\widehat{K})]^d}^2+\|\mathcal E_Kp\|_{H^{m+1}(\widehat{K})}^2\right), \\
    &ii)\ \|\vec{w}\|_{\mathscr F^\Sigma}^2+\sum_{\jj\in J}\|q_\jj\|_{\mathscr F^\Sigma}^2 + \|\vec{v}\|_{\mathscr F^\Sigma}^2+\|q\|_{\mathscr F^\Sigma}^2\\
    &\qquad\lesssim \sum_{K\in\mathscr T_{h,\PP}}h_K^{2m+1}\left(\|\mathcal E_K\vec{w}\|_{[H^{m+1}(\widehat{K})]^d}^2+\sum_{\jj\in J}\|\mathcal E_\jj q_\jj\|_{H^{m+1}(\widehat{K})}^2+\|\mathcal E_K\vec{d}\|_{[H^{m+1}(\widehat{K})]^d}^2+\|\mathcal E_Kp\|_{H^{m+1}(\widehat{K})}^2\right),
\end{aligned}\]
where $\widehat{K}\supseteq K$, for each $K\in\mathscr T_h$, are shape-regular simplexes covering $\mathscr T_h$, thanks to \cref{hp}.
\end{lemma}

\section{{Discontinuous Galerkin method for the Poisson problem of \cref{sec:polyvsstd}}}\label{sec:poisson}
{
This section reports the Discontinuous Galerkin method applied to the Poisson problem
\[\begin{cases}
    -\Delta u=f &\qquad\text{in }\Omega, \\ u=g &\qquad\text{on }\partial\Omega.
\end{cases}\]
We introduce a generic mesh $\mathscr T_h$ of the domain $\Omega$, we collect the element faces in $\mathscr F$, and we employ an analogous notation to \cref{sec:polydg} for jump and average face operators.
The discrete problem reads as follows:\\
Find $u_h\in X_h^\text{DG}$ such that
\begin{equation}
    (\nabla u_h,\nabla v_h)_\Omega - \sum_{F\in\mathscr F}\left(\average{u_h}\cdot\jump{v_h} + \jump{u_h}\cdot\average{v_h} - \frac{\overline{\eta}}{\{h\}_\text{H}}\jump{u_h}\cdot\jump{v_h}\right) = (f,v_h)_\Omega,
\end{equation}
where $\overline\eta$ is a penalty coefficient and $\{h\}_\text{H}$ denotes the harmonic average of the element size across a face.
For further details, we refer the reader to \cite{antonietti2016review}.
We point out that the same formulation can be employed on a standard simplicial or hexahedral mesh, or on a generic polyhedral one.
}

\bibliographystyle{elsarticle-num} 
\bibliography{main}

\end{document}